\documentclass{article}
\usepackage[left=3cm,right=3cm,top=2.5cm,bottom=2.5cm]{geometry} 

\usepackage{amsfonts}
\usepackage{graphicx}
\usepackage{epstopdf}
\usepackage{algorithmic}
\ifpdf
  \DeclareGraphicsExtensions{.eps,.pdf,.png,.jpg}
\else
  \DeclareGraphicsExtensions{.eps}
\fi

\usepackage{fancyhdr}
\title{The Ill-Posed Foundations of Physics-Informed Neural Networks and Their Finite-Difference Variants}

\author{Andreas Langer\thanks{Center for Mathematical Sciences, Lund University, Box 118, 221 00 Lund, Sweden
  (\email{andreas.langer@math.lth.se}, \url{https://portal.research.lu.se/en/persons/andreas-langer/}).}
}

\usepackage{amsopn}

\usepackage{hyperref}
\ifpdf
\hypersetup{
  pdftitle={The Ill-Posed Foundations of PINNs and FD-PINNs},
  pdfauthor={A. Langer}
}
\fi

\usepackage{amsmath}
\usepackage{amsthm}
\usepackage[
    capitalise,
    noabbrev,
  ]{cleveref}
\usepackage[shortlabels, inline]{enumitem}
\usepackage{tikz}
\usetikzlibrary{arrows}
\usepackage[caption=false]{subfig}  % important: caption=false
\usepackage{todonotes}

\newtheorem{theorem}{Theorem}[section]
\newtheorem{example}[theorem]{Example}
\newtheorem{proposition}[theorem]{Proposition}
\newtheorem{remark}[theorem]{Remark}
\newtheorem{corollary}[theorem]{Corollary}

\newtheorem{lemma}[theorem]{Lemma}

\newcommand{\N}{\mathbb{N}}
\newcommand{\Z}{\mathbb{Z}}
\newcommand{\R}{\mathbb{R}}

\newcommand{\mcG}{\mathcal{G}}

\newcommand{\mcD}{\mathcal{D}}
\newcommand{\mcF}{\mathcal{F}}

\newcommand{\mcH}{\mathcal{H}}

\newcommand{\mcJ}{\mathcal{J}}
\newcommand{\mcB}{\mathcal{B}}
\newcommand{\mcA}{\mathcal{A}}

\newcommand{\mcU}{\mathcal{U}}
\newcommand{\mcV}{\mathcal{V}}

\DeclareMathOperator*{\argmin}{arg\,min}
\newcommand{\im}{\mathbf{i}}
\def\FD{\mathrm{FD}}
\newcommand{\nnew}[1]{{#1}}
\newcommand{\email}[1]{\protect\href{mailto:#1}{#1}}

\date{}
\begin{document}

\maketitle

\begin{abstract}
Physics-informed neural networks based on automatic differentiation (AD-PINNs) and their finite-difference counterparts (FD-PINNs) are widely used for solving partial differential equations (PDEs), yet their analytical properties remain poorly understood.
This work provides a unified mathematical foundation for both formulations.
Under mild regularity assumptions on the activation function and for sufficiently wide neural networks of depth at least two, we prove that both the AD- and FD-PINN optimization problems are ill-posed: whenever a minimizer exists, there are in fact infinitely many, and uniqueness fails regardless of the choice of collocation points or finite-difference stencil.
Nevertheless, we establish two structural properties. First, whenever the underlying PDE or its finite-difference discretization admits a solution, the corresponding AD-PINN or FD-PINN loss also admits a minimizer, realizable by a neural network of finite width.
Second, FD-PINNs are tightly coupled to the underlying finite-difference scheme: every FD-PINN minimizer agrees with a finite-difference minimizer on the grid, and in regimes where the discrete PDE solution is unique, all zero-loss FD-PINN minimizers coincide with the discrete PDE solution on the stencil.
Numerical experiments illustrate these theoretical insights: FD-PINNs remain stable in representative forward and inverse problems, including settings where AD-PINNs may fail to converge. We also include an inverse problem with noisy data, demonstrating that FD-PINNs retain robustness in this setting as well.
Taken together, our results clarify the analytical limitations of AD-PINNs and explain the structural reasons for the more stable behavior observed in FD-PINNs.
\end{abstract}

\section{Introduction}

Physics-informed neural networks (PINNs) \cite{LaLiFo:98,RaPeKa:19} are a class of neural networks that incorporate physical laws, typically described by partial differential equations (PDEs), directly into the learning process. 
Unlike traditional machine learning models that rely solely on data, the PINN loss penalizes violations of these laws, allowing the neural network to learn solutions that satisfy the governing equations, even in data-scarce regimes.
This enables them to solve both forward and inverse problems and blends physics-based modeling with data-driven learning by leveraging the power of deep learning. 
In the classical formulation, commonly referred to as a PINN, the differential operators appearing in the PDE are evaluated using automatic differentiation (AD); for clarity, we will refer to this formulation as an AD-PINN.

Compared to classical methods, such as the Finite Element Method, Finite Difference Method, and Finite Volume Method, AD-PINNs offer advantages. They are inherently mesh-free, making them usable for problems involving complex geometries, and they mitigate the curse of dimensionality \cite{DeRyckMishra:24, JeSaWe:21}, allowing efficient handling of high-dimensional PDEs. Their ability to incorporate experimental data directly further enhances their applicability. This allows AD-PINNs to solve naturally inverse problems, e.g., parameter identification in PDEs, via optimization, eliminating the need for computationally expensive techniques such as adjoint methods.

Due to these advantageous, AD-PINNs have garnered significant attention due to their flexibility 
in addressing a wide range of problems involving PDEs. They have been introduced to 
various applications in computational science and engineering, including ﬂuid
mechanics \cite{jin2020nsfnets, raissi2019turbulent, raissi2019vortex, raissi2020hidden, sun2020surrogate}, bio-engineering \cite{costabal2020cardiac, kissas2020cardiovascular}, meta-material design \cite{chen2020nanooptics, fang2019metamaterial, liu2019multifidelity}, free boundary problems \cite{wang2020freeboundary}, Bayesian networks and
uncertainty quantiﬁcation \cite{sun2020bayesian, yang2020bpinns, yang2018deep, yang2019uncertainty, zhu2019physicsconstrained}, high-dimensional PDEs \cite{han2018highdimensional, sirignano2018dgm}, stochastic differential equations \cite{zhang2020modal}, fractional differential equations \cite{pang2020npinns, pang2019fpinns} and many more.

Despite these promising developments, significant limitations remain. AD-PINNs tend to underperform when confronted with complex geometries, intricate boundary conditions, high-frequency components, or multiscale phenomena \cite{raissi2018deep,fuks2020limitations, krishnapriyan2021characterizing, wang2022and}. 
In such settings, they often exhibit instability or fail to converge.
As a consequence of these issues, AD-PINNs frequently struggle to achieve the accuracy of conventional discretization schemes for challenging problems. Recent studies show that AD-PINNs still do not outperform traditional numerical methods such as the finite element method in terms of either accuracy or computational efficiency on common benchmark problems \cite{GrKoLaSc:24}.

The theoretical understanding of PINNs is not yet fully developed. 
Some initial analytical results have been obtained for AD-PINNs, for example, in \cite{DeJaMi:24, DeRyckMishra:22, DoBiBo:23, MishraMolinaro:22, wang2022and}. Under an infinite-width assumption, the AD-PINN training dynamics can be analyzed through the Neural Tangent Kernel, revealing convergence properties and the pronounced imbalance in gradient flow between different loss components \cite{wang2022and}.
In \cite{MishraMolinaro:22} a foundational framework for estimating the error in AD-PINNs is developed, demonstrating its applicability to various PDEs such as viscous scalar conservation laws and the incompressible Euler equations. 
This framework is extended in \cite{DeJaMi:24} to include more detailed error bounds for specific equations like the Navier-Stokes equations, which are fundamental for engineering and fluid dynamics applications. 
In \cite{DeRyckMishra:22} a similar approach is used to provide an error analysis of AD-PINNs for approximating Kolmogorov PDEs, establishing error bounds that depend on the architecture and choice of training points. Their work highlights the conditions under which AD-PINNs can achieve convergence and provides insights into selecting network hyperparameters effectively. 
Complementing these advances, a recent result in \cite{DoBiBo:23} shows that the AD-PINN loss may vanish while the true PDE error remains arbitrarily large, revealing a failure phenomenon. It is also shown that suitable regularization mitigates this effect. 
However, these findings concern specific instances of the AD-PINN methodology and do not resolve more fundamental questions about the well-posedness of PINN optimization problems.

Beyond regularization-based remedies, recent work has explored additional strategies to improve PINN performance, including the following:
\begin{enumerate*}[label=(\roman*)]
\item Innovative training schemes \cite{krishnapriyan2021characterizing, wang2021understanding, wang2022and}, which modify the optimization procedure to improve convergence behavior.
These include a sequence-to-sequence training strategy \cite{krishnapriyan2021characterizing}, a learning rate annealing algorithm \cite{wang2021understanding}, and utilizing the Neural Tangent Kernel framework \cite{wang2022and} in the optimization. 
\item Hybrid approaches \cite{cai2022least, lim2022physics,  pang2019fpinns,  su2024finite, xiang2022hybrid}, which combine neural network formulations with traditional numerical schemes such as finite difference, finite volume, or finite element methods to improve stability, enforce conservation, and handle complex geometries. These methods couple neural representations with discretized PDE operators, bridging data-free PINNs and classical solvers. 
\end{enumerate*}

In this paper we study both AD-PINNs and one such hybrid approach, the finite-difference PINN (FD-PINN), which replaces the continuous differential operators in the PDE residual with finite difference approximations on a discrete grid.
That is, instead of relying on automatic differentiation through the neural network to compute spatial or temporal derivatives, the FD-PINN is trained to satisfy a discretized version of the PDE. Early studies have reported several potential benefits of this strategy. For example, in \cite{JiShZhYa:23} it is suggested that using finite-difference stencils provides more direct derivative estimates, which may help decouple derivative accuracy from the neural network’s approximation error.
Notably, FD-PINNs have achieved marked success in regimes where AD-PINNs faltered. In \cite{JiShZhYa:23}, an FD-PINN with an AD-PINN is compared on the two-dimensional lid-driven cavity flow problem and it is reported that the FD-PINN yields more accurate solutions under identical architectures and training conditions.

While the incorporation of finite difference schemes into PINNs has shown promise, the scope and limits of FD-PINNs are not yet fully understood. The introduction of a mesh and discrete operators raises new questions: for example, how does the choice of grid resolution or difference stencil affect the convergence of the neural network training? What is the trade-off between neural network approximation error and discretization error in the overall solution accuracy? 
 Moreover, FD-PINNs do not entirely escape the pitfalls of AD-PINNs. They still involve training a deep neural network, which means issues like optimization instability or getting stuck in local minima can persist, albeit in modified form. There is currently a lack of theoretical guarantees for FD-PINNs analogous to those being developed for AD-PINNs. Almost none of the existing published results provide convergence rates or stability criteria for FD-PINNs.
 In addition, by introducing a fixed mesh, FD-PINNs may sacrifice some flexibility in handling complex geometries or adaptive refinement compared to mesh-free PINNs. These considerations point to the need for further research to delineate when FD-PINNs will succeed or fail, and how one might optimally design FD-PINN architectures for a given problem.

\paragraph{Contributions of this work}
This work establishes an analytical framework for analyzing the well-posedness of the optimization problems arising from the AD-PINN loss and from the FD-PINN loss. 
We recall that well-posedness is understood in the classical sense of Hadamard: a problem is well-posed if it admits a solution, the solution is unique, and the solution depends continuously on the input data. A violation of any of these conditions renders the problem ill-posed. 
In our setting, the relevant aspects are existence and uniqueness of minimizers of the AD-PINN or FD-PINN loss evaluated on finite collocation sets. Here and in the following, ``minimizer'' always refers to a global minimizer.
While previous studies have documented optimization difficulties and failure modes in PINNs, to the best of our knowledge no existing work has clarified whether these observations reflect an underlying issue of well-posedness. 
Here, we show analytically that both the AD-PINN and FD-PINN optimization problems are ill-posed. Our specific contributions are:
\begin{enumerate}
\item \textbf{Existence of minimizers.}
We prove that whenever the underlying PDE with boundary conditions admits a solution, the associated AD-PINN loss has a minimizer \nnew{within a neural network class}. Moreover, for any fixed finite-difference scheme, if the corresponding finite-difference discretized PDE admits a solution, then the FD-PINN loss constructed from that stencil also \nnew{admits} a minimizer \nnew{within a neural network class}.

\item \textbf{Neural network realizability of minimizers.}
We show that the AD-PINN loss attains the same minimal value whether it is minimized over an unrestricted function space or over depth-2 neural networks of finite width.
\nnew{In particular, our analysis yields an explicit sufficient width such that, for any minimizer of the AD-PINN loss, there exists a depth-2 neural network whose restriction to the collocation set coincides with that minimizer.}
Thus, within this width regime, the attainable accuracy of an AD-PINN is determined solely by the loss construction (collocation points, quadrature, data noise, etc.), rather than by architectural limitations.

\item \textbf{Ill-posedness.} We show analytically that the optimization problems arising from the AD-PINN loss and the FD-PINN loss are ill-posed, admitting non-unique minimizers and, in fact, infinitely many distinct solutions. 

\item \textbf{Equivalence of FD-PINNs and classical finite-difference schemes.}
For any fixed finite-difference discretization and choice of grid points, we show that the corresponding FD-PINN and the finite-difference formulation are equivalent on the grid: every solution of the finite-difference problem can be realized by an FD-PINN that matches it at all grid points, and conversely every FD-PINN minimizer coincides with a finite-difference minimizer on the grid.
In particular, if the discrete PDE has a unique solution and the FD-PINN loss contains no data term incompatible with that solution (so that a zero-loss minimizer exists), then all such FD-PINN solutions agree with the discrete PDE solution on the grid, even though they may differ between grid points.
This identifies FD-PINNs as neural network parameterizations of standard finite-difference schemes and provides a precise connection to classical mesh-based methods (see also \cite[Theorem 4.9]{LangerBehnamian:24} for an analogous equivalence in a variational setting).

\item \textbf{Numerical demonstrations.}
We present three numerical case studies illustrating the practical implications of our analysis:
\begin{enumerate*}[label=(\roman*)]
\item a Poisson problem with challenging boundary geometry, in which AD-PINNs can fail to converge while FD-PINNs successfully recover the solution;
\item a time-dependent Schrödinger equation, serving as a representative forward problem in which FD-PINNs perform comparably to AD-PINNs, and which, to the best of our knowledge, has not previously been treated using FD-PINNs with nonsmooth (ReLU) activations;
%\item a time-dependent Schrödinger equation, showing that FD-PINNs perform comparably to AD-PINNs on a standard forward problem;
\item an inverse problem for the Navier–Stokes equations, demonstrating that FD-PINNs can successfully solve data-driven parameter identification even in the presence of noisy data. 
To the best of our knowledge, FD-PINNs have not previously been applied to inverse problems for PDEs of any kind.
\end{enumerate*}
Since FD-PINNs compute residuals using finite differences rather than AD-based derivatives, they do not require smooth activation functions. We therefore use ReLU activations for all FD-PINN experiments, illustrating that nonsmooth activations are fully admissible in this formulation. By contrast, AD-PINNs rely on smooth activations such as $\tanh$ to ensure well-defined automatic differentiation. 
Prior FD-PINN studies have largely adopted smooth activations for comparability rather than methodological necessity \cite{JiShZhYa:23, RoDuBuSu:24}.
\end{enumerate}

Our results complement and deepen the recent analysis in \cite{DoBiBo:23}. 
In that work, the authors constructed an explicit example (based on the heat 
equation) showing that the AD-PINN loss can vanish while the true PDE error 
becomes arbitrarily large, illustrating a failure mechanism. 
Our findings explain this phenomenon from a different perspective. 
In particular, our ill-posedness results show that non-uniqueness of minimizers is intrinsic to the AD-PINN and FD-PINN optimization problems themselves, and is 
not restricted to any specific PDE model. Thus, the behavior observed in 
\cite{DoBiBo:23} can be interpreted as one concrete instance of a more general 
structural non-identifiability. 
Indeed, our analysis shows that the distance between an AD-PINN minimizer and the true PDE solution can become arbitrarily large, while the loss remains minimal. In this sense, our analysis identifies the underlying optimization-theoretic reasons for both the successes and failures of PINNs, clarifying phenomena that previously appeared to depend on specific PDEs or sampling strategies. 

The rest of the paper is organized as follows: In \cref{Sec:Preliminaries} we introduce the notation, definitions and mathematical framework that underlie the subsequent sections. The AD-PINN formulation is analyzed in \cref{Sec:PINN}, where we identify conditions ensuring the existence of a solution to the associated optimization problem. 
We further prove that, whenever a solution exists, it is never unique, rendering the AD-PINN problem ill-posed. 
Similar results are established in \cref{Sec:FDPINN} for FD-PINNs. More precisely, we show that solutions of the FD-PINN loss coincide exactly with a finite-difference solution on the chosen stencil, while still admitting infinitely many distinct minimizers.  
Thus, although the global FD-PINN optimization problem is also ill-posed, it differs from the AD case in that whenever the discrete PDE admits a unique solution and a zero-loss FD-PINN minimizer exists, all such minimizers induce the same values on the stencil.
In \cref{Sec:NumericalExperiments} we present three numerical experiments illustrating the numerical implications of our analysis and the potential of FD-PINNs. We demonstrate that FD-PINNs can recover the solution to a PDE with challenging boundary geometry in a setting where AD-PINNs can fail, and we further show that FD-PINNs can be used for data-driven parameter identification in PDEs. We conclude with a short discussion in \cref{Sec:Conclusion}.

\section{Preliminaries}\label{Sec:Preliminaries}
This section fixes the notation and mathematical framework for differential operators, PDEs with boundary conditions, and the neural network classes considered later.

\subsection{Definitions and Notations}\label{Sec:Def+Not}

Let $\Omega\subset \R^d$, $d\in\N$, be a bounded domain (i.e., open and connected) with Lipschitz boundary $\partial \Omega$. Denote by $\mcU$ a function space on $\Omega$ and by $\mcV$ a function space on a set $S$, where $S$ is either $\Omega$ or $\partial\Omega$. An operator $\mcG : \mcU \to \mcV$ is called \emph{local} if for all $u,v \in \mcU$ and for all relatively open sets $\tilde{S} \subset S$ (i.e., $\tilde{S} = O \cap S$ for some open $O \subset \R^d$),
\[
u|_{\tilde{S}} = v|_{\tilde{S}} 
\quad \Longrightarrow \quad 
(\mcG u)|_{\tilde{S}} = (\mcG v)|_{\tilde{S}}.
\]
If $S=\Omega$, then the relative topology coincides with the usual topology, 
since $\Omega$ is open in $\R^d$.
Thus ``relatively open in $\Omega$'' simply means an ordinary open subset of $\Omega$. 
If $S=\partial\Omega$, then a set $\tilde{S} \subset \partial\Omega$ is relatively open 
if there exists an open set $O \subset \R^d$ such that 
$\tilde{S} = O \cap \partial\Omega$. 
In this way the same definition applies uniformly to both PDE operators (in the interior) and boundary operators.

In particular, in a PDE setting an operator $\mcG: U \to V$ between two Banach spaces $U,V$ is called a local differential operator of order $r\in\N_0$ if all (weak) derivatives $D^\beta u$ for $|\beta|\leq r$ are well defined for $u\in U$ and belong to $V$, and 
\[
\mcG(u)(x) = G (x, \{D^\beta u (x) \colon |\beta|\leq r\})
\]
for some function $G$. For boundary operators $D^\beta u (x)$ is understood as the trace of $D^\beta u$ at the boundary point $x$. For a sufficiently smooth scalar valued function $\sigma :\R \to \R$ and $k\in\N$ we denote by $\sigma^{(k)}$ its $k$-th continuous derivative.

Besides locality, we will also use a few standard notions from analysis.
A function $f:\Omega \to \R$ is called \emph{continuous and piecewise affine} if there exists a finite partition of polyhedra that cover $\Omega$ and $f$ is affine on each polyhedron and continuous in $\Omega$.
For a Banach space $U$ we denote its associated norm by $\|\cdot\|_U$.
A functional $\mcJ:U \to \overline{\R}:=\R \cup \{\pm\infty\}$ is said to be \emph{coercive}, if $\|u_n\|_U \to \infty$ implies $\mcJ(u_n) \to \infty$ for any sequence $(u_n)_n\subset U$. It is called \emph{\nnew{(weakly)} lower semicontinuous} if for all $u \in U$ we have that $\liminf_{n \to \infty} \mcJ(u_n) \ge \mcJ(u)$ for any sequence $(u_n)_n \subset U$ converging \nnew{(weakly)} to $u$ as $n \to \infty$.

\subsection{Problem}

Let $\Gamma \subseteq \partial\Omega$ denote the portion of the boundary on which boundary conditions are prescribed.
We consider an unknown field
\(
u : \overline{\Omega} \to \R^c,
\)
where $c\in\N$ denotes the number of components.
To describe the interior equation and the boundary conditions, we fix Banach spaces
\[
U \subset \{u:\overline{\Omega}\to\R^c\}, \qquad
V \subset \{v:\Omega\to\R^{c_\mcF}\}, \qquad
W \subset \{w:\Gamma\to\R^{c_\mcB}\},
\]
where $c_\mcF,c_\mcB\in\N$ allow for general systems of equations.  
The space $U$ represents the space of admissible solution functions, while $V$ and $W$ represent the ranges of the interior and boundary operators, respectively. 
We assume that $U$, $V$, and $W$ are continuously embedded into $L^\nu(\Omega,\R^c)$, $L^\nu(\Omega,\R^{c_\mcF})$, and $L^\nu(\Gamma,\R^{c_\mcB})$, respectively, for some fixed $\nu\in[1,\infty)$, so that they admit consistent discretizations via empirical $\ell^\nu$-norms when collocation points are introduced later.

A \emph{partial differential equation with boundary conditions} is given by
\begin{equation}\label{PhysicalModel}
\begin{split}
\mcF(u)(z) &= 0 \qquad z\in\Omega,\\
\mcB(u)(z) &= 0 \qquad z\in\Gamma\subseteq\partial\Omega,
\end{split}
\end{equation}
where \(u\in U\) is the unknown.
Here \(\mcF:U\to V\) is a (possibly nonlinear) differential operator of order \(r_{\mcF}\in\N_0\) describing the interior equation, and \(\mcB:U\to W\) is a differential operator of order \(r_{\mcB}\in\N_0\) describing the boundary condition.
By a ``differential operator'' we mean that, for $\mcF$, the value $\mcF(u)(z)$ depends only on $z$ and the derivatives $D^\beta u(z)$ with $|\beta|\le r_{\mcF}$. 
Likewise, a boundary operator of order $r_{\mcB}$ depends only on $z\in\Gamma$ and the derivatives $D^\beta u(z)$ with $|\beta|\le r_{\mcB}$.
Thus both $\mcF$ and \(\mcB\) are local differential operators in the sense introduced in \cref{Sec:Def+Not}.
That is, there exist functions $F$ and $B$ such that, for $u\in U$ and $z\in\Omega$,
\begin{equation}\label{eq:F}
\mcF(u)(z)
= F\left(z,\{D^\beta u (\nnew{z}) \colon |\beta|\leq r_\mcF\}\right),
\end{equation}
and similarly, for $z\in\Gamma$,
\begin{equation}\label{eq:B}
\mcB(u)(z)
= B\left(z,\{D^\beta u (\nnew{z}) \colon |\beta|\leq r_\mcB\}\right).
\end{equation}
Since we consider \eqref{PhysicalModel} in its strong form only, to ensure that all quantities are well-defined pointwise, the solution space $U$ is required to satisfy
\[
U \subseteq C^{r_{\mcF}}(\Omega,\R^c) \cap C^{r_{\mcB}}(\overline{\Omega},\R^c).
\]
Since $\Omega$ is bounded and $C^{r_{\mcF}}(\Omega,\R^c) \cap C^{r_{\mcB}}(\overline{\Omega},\R^c) \subset C(\overline{\Omega}\nnew{, \R^c})$, every $u\in U$ is bounded and the continuous embedding of $U$ in $L^\nu(\Omega\nnew{, \R^c})$ holds automatically for all $1\leq \nu \leq \infty$; see \cite[2.14 Theorem]{AdamsFournier:03}.

\subsection{Neural Networks}\label{Sec:NN}

We use fully connected feedforward neural networks, simply referred to as \emph{neural networks}, with elementwise (componentwise) activation
\(\sigma:\R\to\R\) in all hidden layers and a linear output layer.
Let \(L\in\N\) denote the depth, with input dimension \(d_0=d\) and output dimension \(d_L=c\), matching the number of components of the unknown field $u:\overline{\Omega} \to \R^c$. A neural network realizes a function \(f:\R^{d_0}\to\R^{d_L}\) of the form
\begin{equation*}
\begin{split}
&\varphi_{0}(x)=x,\\
&\varphi_{i}(x)=\sigma\!\left(W_{i}\varphi_{i-1}(x)+b_{i}\right)\ \ \text{for }i=1,\dots,L-1,\\
&f(x)=W_{L}\varphi_{L-1}(x)+b_{L},
\end{split}
\end{equation*}
with weights \(W_{i}\in\R^{d_i\times d_{i-1}}\) and biases \(b_{i}\in\R^{d_i}\), where \(d_i \in \N\) for \(i=0, \ldots, L\in\N\).

Let $\mathcal{N}_L$ denote the class of depth-$L$ neural networks 
$f:\R^d \to \R^c$ with arbitrary hidden-layer widths.  
We define the associated hypothesis space
\begin{equation*}
    \mcH
    := \{\, f|_{\overline{\Omega}} : f \in \mathcal{N}_L \}
    \subseteq \{\, u : \overline{\Omega} \to \R^c \}.
\end{equation*}
Thus, $\mcH$ consists of all functions realizable by depth-$L$ neural networks
when restricted to $\overline{\Omega}$, since only the behavior on $\overline{\Omega}$ enters the PDE formulation. 
Allowing arbitrary widths is important, as the resulting class $\mcH$ is then closed under linear combinations (see \cref{sec:closure}), a property used later in the analysis.

For a given choice of layer widths 
$(d_0,\dots,d_L)$, the total number of parameters (weights and biases) is
\[
M = \sum_{i=0}^{L-1} d_{i+1}\,(d_i + 1).
\]
We then define $\mcH^M$ as the subset of $\mcH$ consisting of all functions realizable by depth-$L$ neural networks with exactly $M$ parameters. 
In particular, $\mcH^M \subset \mcH$ and $\mcH$ is the union of 
$\mcH^M$ over all admissible architectures.
The collection of all weights and biases of a neural network is denoted by $\theta\in\R^M$. In the sequel for a neural network $f$ we will often write $f_\theta$ to express its dependency on the weights and biases.

\paragraph{Differentiability}
Whenever derivatives $D^\beta f$ of the neural network output are required (for example when evaluating differential operators of order $r_{\mcF}$ or $r_{\mcB}$), their existence must be guaranteed both in the interior and, for $|\beta|\le r_{\mcB}$, on the boundary.
This is ensured whenever the activation function $\sigma$ is $C^{r_{\max}}$ with  $r_{\max} := \max(r_{\mcF}, r_{\mcB})$, since all derivatives $D^\beta f$ with $|\beta|\le r_{\max}$ then exist classically. 
Typical smooth activations satisfying this regularity assumption include the sigmoid, $\tanh$, softplus, GELU, and other $C^\infty$ functions, all of which ensure that the required classical derivatives exist.

However, we will frequently work with the rectifier linear unit (ReLU) activation $\sigma(x)=\max\{0,x\}$. By ReLU-NN we denote a neural network whose activation functions are ReLUs. Although ReLU-NNs do not belong to the classical spaces $C^r(\overline{\Omega})$, $r\in\N$, required for our strong PDE formulation, the AD-PINN loss only requires pointwise evaluations of a function $f_\theta$ and its derivatives at a finite set of collocation points. 
Since ReLU-NNs are differentiable except at finitely many kink locations, we restrict attention to the subclass
\[
\mcH_{\textrm{reg}}:=\{f_\theta \in \mcH \colon f_\theta \text{ is } C^{r_{\max}}\text{ at all collocation points} \},
\]
on which the AD-PINN loss is fully well-defined. This explains why ReLU could sometimes be still used in practice for AD-PINNs, even though it does not belong to the classical function space $U$. Analogues to above we define by $\mcH_{\textrm{reg}}^M \subset \mcH_{\textrm{reg}}$ all depth-$L$ neural networks in $\mcH_{\textrm{reg}}$ with exactly $M$ parameters.

A similar restriction is unnecessary for FD-PINNs: the FD-PINN loss involves only finite-difference stencils and therefore does not require continuous differentiability of a neural network. Consequently, ReLU-NNs can be used for FD-PINNs without any additional regularity assumptions at the collocation points.

A key structural property of ReLU-NNs is that ReLU-NNs of smaller depth can be embedded exactly into deeper ReLU-NNs: $x = \max\{0,x\} - \max\{0,-x\}$ for all $x\in\R$ provides an explicit construction of the identity map, allowing one to insert pairs of layers without changing the realized function. That is, any depth-$\tilde{L}$ ReLU-NN with $\tilde{L}<L$ can also be realized exactly by a deeper neural network of depth $L$, which we will make use of in our theory.

\section{AD-PINN Framework}\label{Sec:PINN}

To solve \eqref{PhysicalModel} in an AD-PINN framework, we first introduce a 
continuous loss functional that measures violation of the PDE and its 
boundary conditions and then minimize this functional over a class \(\mcH \subseteq U\) 
of depth-\(L\) neural networks introduced in \cref{Sec:NN}. This leads to the optimization problem 
\begin{align}\label{eq:continuousPinn}
\min_{u_\theta\in \mcH} \{\mcJ(u_\theta):=\alpha_{\mcF}\|\mcF(u_\theta)\|_V + \alpha_{\mcB}\|\mcB(u_\theta)\|_W\}
\end{align}
with weights \(\alpha_{\mcF},\alpha_{\mcB}\ge 0\), which we refer to as the \emph{continuous PINN}, since the objective is defined through Banach-space norms of the residuals.
In this continuous setting, the loss $\mcJ$ consists solely of the physics- and boundary-based terms;
observational data, when available, will only be incorporated later when we discuss the AD-PINN, i.e., the formulation in which these functional norms are approximated by quadrature on a finite set of collocation points while retaining analytical derivatives.

Note that $u\in U$ solves \eqref{PhysicalModel} if and only if $\mcJ(u)=0$, since $\mcF(u)=0$ and $\mcB(u)=0$ precisely characterize solutions of \eqref{PhysicalModel}. In particular, if $u_\theta\in\mcH\subseteq U$ such that $\mcJ(u_\theta)=0$, which implies that $u_\theta$ is a minimizer of \eqref{eq:continuousPinn}, then $u_\theta$ solves also \eqref{PhysicalModel} and is a solution of $\argmin_{u\in U} \mcJ(u)$.

\subsection{Existence of Minimizers}
Classical universal approximation theorems \cite{hornik1991approximation,LeLiPiSch:93} ensure that neural networks with non-polynomial activations are dense in $C(\overline{\Omega}, \R^c)$ and in $L^{\nu}(\Omega,\R^c)$ for $1 \le \nu < \infty$, and, if the activation is $r$-times continuously differentiable and non-polynomial, also dense in $C^{r}(\overline{\Omega},\R^c)$ for any finite $r \ge 1$. 
Hence any sufficient regular solution of \eqref{PhysicalModel} can be approximated arbitrary well by elements of $\mcH$.
However, density alone does not imply that \eqref{eq:continuousPinn} admits a minimizer.
Depending on the choice of $\mcH$ and on the structure of \eqref{PhysicalModel}, the variational problem \eqref{eq:continuousPinn} may fail to attain its infimum even when \eqref{PhysicalModel} itself has a classical solution.
For instance, suppose $\mcH$ consists of neural networks with smooth, real-analytic activations such as $\tanh$ \cite{NguyenHein2017}. Then every $u\in\mcH$ is real-analytic on $\overline{\Omega}$ \cite{NguyenHein2017}, whereas there exist problems \eqref{PhysicalModel} whose unique classical solution $\hat{u}$ lies in $C^1(\overline{\Omega},\R^c)$ but is not real-analytic. By universal approximation results, one can find a sequence $(u_n)_n \subset \mcH$ with $\mcJ(u_n) \to \mcJ(\hat{u})$ for $n\to\infty$, but the infimum of \eqref{eq:continuousPinn} over $\mcH$ is not attained.

This phenomenon can also be understood from an optimization-theoretic perspective via the Weierstraß theorem. The theorem guarantees existence of a minimizer if $\mcJ$ is lower semicontinuous and if the level set 
$\{u_\theta \in\mcH \colon \mcJ(u_\theta)\leq \mcJ(\bar u_\theta)\}$ is compact for some $\bar u_\theta\in \mcH$. 
\nnew{These assumptions hold, for example, if $U$ is a reflexive Banach space, $\mcJ$ is coercive and weakly lower semicontinuous, and the characteristic function $\chi_{\mcH}: U \to \overline{\R}$, defined by $\chi_{\mcH}(u)=0$ for $u\in\mcH$ and $\chi_{\mcH}(u)=+\infty$ otherwise, is weakly lower semicontinuous.}
However, the \nnew{weak} lower semicontinuity of $\chi_{\mcH}$ may fail, precisely as in the analytic-activation example above, and \nnew{weak} compactness of the level sets is generally not guaranteed. Consequently, one cannot expect \eqref{eq:continuousPinn} to admit a minimizer in general.

To address the possible lack of compactness of level sets of $\mcJ$ on $\mcH$, one may restrict the admissible neural networks to a parameter-bounded subset
\[
\mcH_{c_\theta}^M := \{\,u_\theta \in \mcH^M : |\theta|_{q} \le c_\theta\,\},
\qquad c_\theta\ge 0,\; q\in\N\cup\{\infty\},
\]
where $|\cdot|_q$ denotes the standard $\ell^q$-norm.
This leads to the constrained problem
\begin{equation}\label{eq:continuousPinn:c}
    \min_{u_\theta \in \mcH_{c_\theta}^M} \mcJ(u_\theta)
    = 
    \min_{\theta\in \R^M, |\theta|_{q}\le c_\theta}\mcJ(u_\theta).
\end{equation}
where $c_\theta$ is typically chosen large such that \eqref{eq:continuousPinn:c} is a close approximation to \eqref{eq:continuousPinn}. 
The constraint $|\theta|_q \leq {c_\theta}$ with ${c_\theta}<\infty$ ensures that $\theta$ stays bounded and renders the solution space $\{\theta \in \R^M \colon |\theta|_q\leq {c_\theta}\}$ compact, where $M\in\N$ is fixed. 
Assuming that $\mcF$, $\mcB$, and the activation $\sigma$ are continuous, the existence of a minimizer of \eqref{eq:continuousPinn:c} follows 
by similar arguments as in the proof of \cite[Theorem 3.4]{LangerBehnamian:24}. 
The restriction to finite $M$ is essential, because without a finite number of weights and biases the parameter space is non-compact even under a norm bound \cite[2.5-5 Theorem]{Kreyszig1991}.

Instead of imposing a hard constraint on $\theta$, one may incorporate a so-called ridged regularization \cite{DoBiBo:23}, i.e., a norm penalty on $\theta$ into the objective:
\begin{equation}\label{eq:continuousPinn:ridged}
\min_{u_\theta\in \mcH^M} \mcJ(u_\theta) + \alpha_{\theta} |\theta|_q=  \min_{\theta \in \R^M}\mcJ(u_\theta) + \alpha_{\theta} |\theta|_q,
\end{equation}
where $\alpha_{\theta}> 0$. Under the same continuity assumptions on $\mcF$, $\mcB$, and the activation function
$\sigma$, the map $\theta \mapsto \mcJ(u_\theta)$ is continuous. Since the ridge term
$\alpha_\theta |\theta|_q$ is coercive on $\R^M$, the full objective in
\eqref{eq:continuousPinn:ridged} is continuous and coercive,  and therefore
attains its minimum.

As \eqref{eq:continuousPinn}, \eqref{eq:continuousPinn:c} and \eqref{eq:continuousPinn:ridged} involve Banach space norms, they require integration over $\Omega$ and $\Gamma$ and cannot be evaluated exactly in practice. 
In the discrete setting, referred to as AD-PINNs, the continuous formulations \eqref{eq:continuousPinn}, \eqref{eq:continuousPinn:c} and \eqref{eq:continuousPinn:ridged} are replaced by finite-sum minimization problems over collocation points. 
We denote by $\Omega^h\subset\Omega$ and $\Gamma^h\subset\Gamma$ the sets of interior and boundary collocation points, respectively, and by $\mcD^h\subseteq \Omega^h \cup \Gamma^h$ the locations of observation data $u^*$. 
We associate quadrature weights $\omega_\mcF^z,\omega_\mcB^z,\omega_\mcD^z$ with these sets, typically chosen as $\omega_\mcF^z = 1/|\Omega^h|$, $\omega_\mcB^z = 1/|\Gamma^h|$, $\omega_\mcD^z = 1/|\mcD^h|$.
To allow for functions that are not globally $C^{r_{\mcF}}$ or $C^{r_{\mcB}}$ (such as ReLU-NNs), we introduce the node-regular space
\[
U^h
:= 
\left\{u:\overline\Omega\to\R^c \colon D^\beta u \text{ exists classically on }\Omega^h (|\beta|\le r_{\mcF}) \text{ and on }\Gamma^h (|\beta|\le r_{\mcB}) \right\}.
\]
By construction $U\subset U^h$, and any $\mcH$ consisting of functions that are continuously differentiable up to order $r_{\mcF}$ on $\Omega^h$ and up to order $r_{\mcB}$ on $\Gamma^h$, for example the class $\mcH_{\mathrm{reg}}$ with ReLU activations introduced in \cref{Sec:NN}, satisfies $\mcH\subset U^h$. 

We use the local expressions \eqref{eq:F} and \eqref{eq:B} to extend the evaluation of $\mcF$ and $\mcB$ to $U^h$. For $u\in U^h$ and $z\in\Omega^h$ we define
\[
\widehat{\mcF}(u)(z)
:=
F\left(z,\{D^\beta u (\nnew{z}) \colon |\beta|\leq r_\mcF\}\right),
\]
and for $z\in\Gamma^h$ we define
\[
\widehat{\mcB}(u)(z)
:=
B\left(z,\{D^\beta u (\nnew{z}) \colon |\beta|\leq r_\mcB\}\right).
\]
For $u\in U$ these definitions agree with the original operators at the
collocation points,
\[
\widehat{\mcF}(u)(z) = \mcF(u)(z),\quad z\in\Omega^h,
\qquad
\widehat{\mcB}(u)(z) = \mcB(u)(z),\quad z\in\Gamma^h.
\]
In particular, $\widehat{\mcF}$ and $\widehat{\mcB}$ provide a well-defined discrete residual for all $u\in U^h$, including nonsmooth functions such as ReLU-NNs, provided the required derivatives exist at the collocation points.
For $\nu\in[1,\infty)$ the AD-PINN functional is then defined on $U^h$ by
\begin{equation}\label{eq:discretePINN}
\begin{split}
\mcJ^h(u) 
:= \alpha_{\mcF}\sum_{z\in\Omega^h}
 \omega_{\mcF}^z \, \left|\widehat{\mcF}(u)(z)\right|_\nu^\nu 
&+ \alpha_{\mcB}\sum_{z\in\Gamma^h}
 \omega_{\mcB}^z \, \left|\widehat{\mcB}(u)(z)\right|_\nu^\nu  \\
&+ \alpha_{\mcD}\sum_{z\in\mcD^h}  
 \omega_{\mcD}^z \, \left|u(z) - u^*(z)\right|_\nu^\nu,
 \end{split}
\end{equation}
where $\alpha_\mcD\geq 0$.
For $u\in U$ this coincides with the AD-PINN functional obtained
by using $\mcF(u)(z)$ and $\mcB(u)(z)$ in place of 
$\widehat{\mcF}(u)(z)$ and $\widehat{\mcB}(u)(z)$.
With this notation, the AD-PINN problems read as
\begin{align}
\text{(unconstrained)} &\quad 
\min_{u \in \mcH} \; \mcJ^h(u), 
\label{eq:discretePINN:plain} \\
\text{(constrained)} &\quad 
\min_{u \in \mcH_{c_\theta}^M} \; \mcJ^h(u),
\label{eq:discretePINN:c} \\
\text{(regularized)} &\quad 
\min_{\theta \in \R^M} \; \mcJ^h(u_\theta) + \alpha_\theta |\theta|_q.
\label{eq:discretePINN:rigid}
\end{align}

While for \eqref{eq:continuousPinn} we cannot guarantee the existence of a minimizer, even when \eqref{PhysicalModel} itself has a classical solution, the situation changes after discretization. Although the Weierstra\nnew{ß} theorem does not need to hold for \eqref{eq:discretePINN:plain}, since compactness and lower semicontinuity issues from the continuous setting may persist, the discrete functional $\mcJ^h$ is nevertheless more favourable and existence will follow from a different structural argument that we develop below. To prepare for this existence result, we first establish the following structural property:
for depth-2 neural networks with sufficiently smooth activations, every minimizer of the AD-PINN loss in $U$ can be realized by a finite-width neural network, and conversely every finite-width AD-PINN minimizer is also a minimizer over $U$.
\begin{proposition}\label{thm:MinimizerCoincide}
Let $U \subseteq C^{r_\mcF}(\Omega,\R^c) \cap C^{r_\mcB}(\overline{\Omega},\R^c)$.  
Consider finite collocation sets 
\(
\Omega^h = \{z_\mcF^i\}_{i=1}^{N_\mcF} \subset \Omega\) and \( \Gamma^h = \{z_\mcB^j\}_{j=1}^{N_\mcB} \subset \Gamma\) with $N_\mcF,N_\mcB\in\N$ and set $\ell:= N_\mcF (d +r_\mcF) + N_\mcB(d+r_\mcB)$. 
Let $\mcH^M$ denote a class of depth-$2$ neural networks with  $c\binom{\ell}{d}$ hidden units (i.e., $M = c\binom{\ell}{d}(d+1) + c(c\binom{\ell}{d} + 1)$) and activation $\sigma \in C^{\ell-d}(\R,\R)$, satisfying $\sigma^{(k)}(a) \neq 0$ for some $a\in\R$ and all $0 \le k \le \ell-d$. 
Then we have that 
\begin{enumerate}[(i)]
\item if $\hat{u}\in \argmin_{u \in U} \mcJ^h(u)$, then there is $u_\theta\in \mcH^M$ that minimizes $\mcJ^h$ over $\mcH^M$ 
with $\mcJ^h(u_\theta) = \mcJ^h(\hat{u})$ and $u_\theta(z) = \hat{u}(z)$ for all $z\in\Omega^h\cup\Gamma^h$;
\item if $\hat{u}_\theta\in \argmin_{u_\theta \in \mcH^M} \mcJ^h(u_\theta)$, then $\hat{u}_\theta \in \argmin_{u \in U} \mcJ^h(u)$.
\end{enumerate}
\end{proposition}

\begin{proof}
\begin{enumerate}[(i)]
\item Let $\hat{u}\in \argmin_{u \in U} \mcJ^h(u)$. Then by \cref{cor:pointwise-bc-vector} there is a Hermite interpolant $u_\theta\in \mcH^M$ of $\hat{u}$ such that
\begin{equation}\label{Eq:HermiteEquations}
\begin{split}
D^\beta u_\theta(z_\mcF^i)=D^\beta \hat{u}(z_\mcF^i)\quad\text{for all }i=1,\dots,N_\mcF,\;|\beta|\le r_\mcF,\\
D^\gamma u_\theta(z_\mcB^j)=D^\gamma \hat{u}(z_\mcB^j)\quad\text{for all }j=1,\dots,N_\mcB,\;|\gamma|\le r_\mcB,
\end{split}
\end{equation}
which implies that $\mcJ^h(u_\theta) = \mcJ^h(\hat{u})$. Assume there exists $\tilde{u}_\theta\in\mcH^M$ with $\mcJ^h(\tilde{u}_\theta) < \mcJ^h(u_\theta)$. Note that functions in $\mcH^M$ have at least regularity $C^{r_\mcF + r_\mcB + d}$, since for $N_\mcF = N_\mcB=1$ we obtain $\ell=r_\mcF + r_\mcB + 2d$, and hence $\mcH^M \subseteq U$. 
Consequently $\tilde{u}_\theta \in U$ which contradicts the optimality of $\hat{u}$ in $U$ and hence $u_\theta \in \argmin_{u \in \mcH^M} \mcJ^h(u)$.

\item Let $\hat{u}_\theta\in \argmin_{u \in \mcH^M} \mcJ^h(u)$ and note that $\hat{u}_\theta\in U$, since $\mcH^M\subseteq U$. 
Assume there is a ${u}\in U$ such that $\mcJ^h({u}) < \mcJ^h(\hat{u}_\theta)$. 
Use \cref{cor:pointwise-bc-vector} to construct a Hermite interpolant ${u}_\theta \in \mcH^M$ with properties \eqref{Eq:HermiteEquations}. 
Then $\mcJ^h({u}_\theta) = \mcJ^h(u)< \mcJ^h(u_\theta)$, which contradicts the optimality of $\hat{u}_\theta$. Hence $\hat{u}_\theta \in \argmin_{u \in U} \mcJ^h(u)$.
\end{enumerate} 
\end{proof}
In words, \cref{thm:MinimizerCoincide}~(i) says that if the AD-PINN functional $\mcJ^h$ attains its minimum over the full space $U$, then the depth-2 neural network class $\mcH^M$ specified in \cref{thm:MinimizerCoincide} is expressive enough to contain at least one minimizer. \cref{thm:MinimizerCoincide}~(ii) states the converse inclusion: any minimizer over $\mcH^M$ is already a minimizer over $U$. 
Thus, under the stated assumptions, although the full class $\mcH$ may contain additional minimizers, the restricted class $\mcH^M$ is nevertheless guaranteed to contain at least one minimizer of $\mcJ^h$.
The following corollary makes this connection to the AD-PINN problem
\eqref{eq:discretePINN:plain} explicit.
\begin{corollary}\label{cor:solution:ADPINN}
Let the assumptions and notations of \cref{thm:MinimizerCoincide} hold, and assume that the AD-PINN problem \eqref{eq:discretePINN:plain} admits at least one minimizer in $\mcH$. Then there exists a neural network $u_\theta \in \mcH^M$ such that $u_\theta$ solves \eqref{eq:discretePINN:plain}.
\end{corollary}
\begin{proof}
Let $\hat{u}\in\mcH$ be a solution of \eqref{eq:discretePINN:plain}. By \cref{cor:pointwise-bc-vector} there exists a single-hidden layer neural network $u_\theta\in\mcH^M$ with $c\binom{\ell}{d}$ hidden units that interpolates $\hat{u}$ such that \eqref{Eq:HermiteEquations} holds. Hence $\mcJ^h(u_\theta)=\mcJ^h(\hat{u})$ and $u_\theta$ solves \eqref{eq:discretePINN:plain}. 
\end{proof}

Utilizing \cref{thm:MinimizerCoincide} we are able to show that if \eqref{PhysicalModel} has a solution in $U$, then also \eqref{eq:discretePINN:plain} has a solution.

\begin{theorem}
\label{thm:equivalent:solution}
Let the assumptions and notations of \cref{thm:MinimizerCoincide} hold.
If $\hat{u} \in U$ is a solution of~\eqref{PhysicalModel} and $\alpha_\mcD=0$, then there exists a one-hidden-layer neural network 
$u_\theta \in {\mcH}^M \subset {\mcH}$ such that $u_\theta$ solves \eqref{eq:discretePINN:plain} and $\mcJ^h(u_\theta) = 0$.
\end{theorem}
\begin{proof}
For a solution $\hat{u}\in U$ of \eqref{PhysicalModel} we have $\mcJ(\hat{u})=0$ and consequently $\hat{u}\in\argmin_{u\in U} \mcJ(u)$. Since $\alpha_\mcD =0$, we even obtain that $\mcJ(\hat{u}) = \mcJ^h(\hat{u})$ and hence $\hat{u}\in\argmin_{u\in U}\mcJ^h(u)$. \Cref{thm:MinimizerCoincide}~(i) implies then the existence of a $u_\theta\in\mcH^M \subset \mcH$ such that $u_\theta\in\argmin_{u\in \mcH^M} \mcJ^h(u)$ and $\mcJ^h(u_\theta) = \mcJ^h(\hat{u})=0$.
\end{proof}

\begin{figure}
\begin{center}
\begin{tikzpicture}[
    >=stealth,
    node distance=2.5cm,
    formula/.style={draw, rounded corners, inner sep=4pt}
]
% Nodes 
\node[formula, align=center] (f0) at (0, 3.2)  {\footnotesize{$\hat{u}\in U$ solves \eqref{PhysicalModel}}}; % top left
\node[formula, align=center] (f1) at (0, 1.5)  {\footnotesize{$\hat{u}\in \argmin_{u\in U} \mcJ(u)$}\\[1pt] \footnotesize{with $\mcJ(\hat{u})=0$}}; % top left
\node[formula, align=center] (f2) at (8, 1.5)  {\footnotesize{$\hat{u}\in \argmin_{u\in \mcH} \mcJ(u)$}\\[1pt] \footnotesize{with $\mcJ(\hat{u})=0$}}; % top right
\node[formula, align=center] (f3) at (0,-1.5)  {\footnotesize{$\hat{u}\in \argmin_{u\in U} \mcJ^h(u)$}\\[1pt] \footnotesize{with $\alpha_\mcD=0$}}; % bottom left
\node[formula, align=center] (f4) at (8,-1.5)  {\footnotesize{$\hat{u}\in \argmin_{u\in \mcH} \mcJ^h(u)$}\\[1pt] \footnotesize{with $\alpha_\mcD=0$}}; % bottom right

% Horizontal connections
\draw[implies-,double equal sign distance, thick,shorten >=2pt,shorten <=2pt] (f1) -- node[midway, above=-1pt]{\footnotesize{$\mcH \subseteq U$}} (f2);
\draw[implies-,double equal sign distance, thick,shorten >=2pt,shorten <=2pt] (f3) -- node[midway, above=-1pt]{\footnotesize{\cref{thm:MinimizerCoincide}~(ii)}}(f4);
% New arrow slightly below it
\draw[->,thick,shorten >=2pt,shorten <=2pt]
  ([yshift=-5pt]f3.east)
    -- node[midway, below=-1pt] {\footnotesize{$\exists \hat{u} \in \mcH$ (\cref{thm:MinimizerCoincide}~(i))}}
  ([yshift=-5pt]f4.west);
  
% Vertical connections
\draw[implies-implies,double equal sign distance, thick,shorten >=2pt,shorten <=2pt] (f0) -- (f1);
\draw[-implies,double equal sign distance, thick,shorten >=2pt,shorten <=2pt] (f1) -- node[midway, above=-1pt, sloped]{\footnotesize{$\mcJ^h(\hat{u})=0$}} (f3);
\draw[-implies,double equal sign distance, thick,shorten >=2pt,shorten <=2pt] (f2) -- node[midway, above=-1pt, sloped]{\footnotesize{$\mcJ^h(\hat{u})=0$}}(f4);

% Diagonal connections
\draw[->, thick, shorten >=2pt,shorten <=2pt] (f1) -- node[midway, above=-1pt, sloped] {\footnotesize{$\exists \hat{u}\in\mcH$}} node[midway, below=-1pt, sloped] {\footnotesize{\cref{thm:equivalent:solution}}}(f4);
\end{tikzpicture}
\end{center}
\caption{Schematic overview of the four minimization problems associated with the continuous and discrete functionals $\mcJ$ and $\mcJ^h$, posed over the full space $U$ or over the neural network class $\mcH$, and their relation to the underlying PDE. Double arrows ($\Rightarrow$) indicate logical implications, either holding by definition or proved in \cref{thm:MinimizerCoincide}. Single arrows ($\to$) denote the existence-type relations established in \cref{thm:MinimizerCoincide,thm:equivalent:solution}.}
\label{fig:minimizer-structure}

\end{figure}
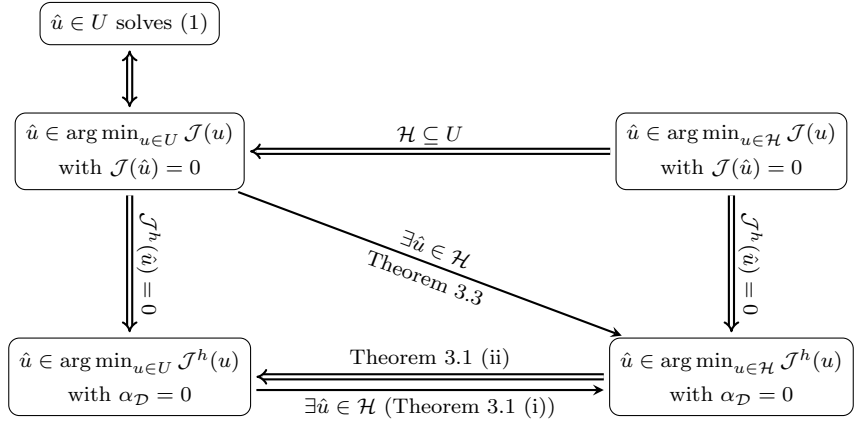

\Cref{fig:minimizer-structure} summarizes the logical relations among the key objects in the AD-PINN formulation:
the continuous PDE \eqref{PhysicalModel}, the continuous loss functional $\mcJ$, its discrete counterpart $\mcJ^h$, and their restrictions to the neural network class $\mcH$.
The schematic visualizes exactly the implication structure proved in \cref{thm:MinimizerCoincide,thm:equivalent:solution} (in the regime $\alpha_\mcD=0$), together with the straightforward inclusions that follow directly from the definitions.
It therefore provides a compact overview of how solutions of the PDE relate to minimizers of the continuous PINN and AD-PINN objectives, and how these minimizers behave when the solution space is restricted from the full function space $U$ to the neural network class $\mcH$.

While \cref{thm:MinimizerCoincide,cor:solution:ADPINN} hold for all choices of $\alpha_\mcD\ge 0$, the situation is different for \cref{thm:equivalent:solution}. The reason is that if the observed data are noisy, the exact solution $\hat{u}\in U$ of \eqref{PhysicalModel} will in general not minimize $\mcJ^h$ over $U$, and therefore the conclusion of \cref{thm:equivalent:solution} need not hold when $\alpha_\mcD>0$.
However, in the noise free case, that is, when 
$u^*(z)=\hat{u}(z)$ for all $z\in\mcD$, the last term in \eqref{eq:discretePINN} vanishes at $\hat{u}$. Consequently, \cref{thm:equivalent:solution} remains valid also for $\alpha_\mcD > 0$ in this setting.

For \eqref{eq:discretePINN:c} and \eqref{eq:discretePINN:rigid}, the existence of a solution follows by the same argument as for \eqref{eq:continuousPinn:c} and \eqref{eq:continuousPinn:ridged}, provided that $\mcF$, $\mcB$, and $\sigma$ are continuous. 
However, analogues of \cref{thm:MinimizerCoincide,thm:equivalent:solution} appear more challenging for \eqref{eq:discretePINN:c} and \eqref{eq:discretePINN:rigid}, as one must carefully handle the bound $c_\theta$ and the penalization term $\alpha_\theta |\theta|_q$, both of which influence the magnitude of the weights and biases and hence the solutions. A systematic treatment of these cases seems more difficult and is left for future work.

\subsection{Non-uniqueness of Minimizers}

While under certain assumptions the existence of a solution of \eqref{eq:discretePINN:plain} can be shown, see \cref{thm:MinimizerCoincide}~(i) and \cref{thm:equivalent:solution}, and for \eqref{eq:discretePINN:c} and \eqref{eq:discretePINN:rigid} under mild continuity assumptions, the question of uniqueness is far more delicate. 

 We present a simple example illustrating that even if \eqref{PhysicalModel} has a unique solution the respective AD-PINN optimization problem \eqref{eq:discretePINN:plain} does not necessarily have a unique solution, but infinitely many.

\begin{example}\label{Example:PINN}
Consider the 1D differential equation
\begin{equation}\label{eq:PINN:counterexample1}
\begin{split}
&u'(z) = a \in\R \quad \text{for }z\in(0,T)\subset \R,\\
&u(0) = u_0\in\R.
\end{split}
\end{equation}
The unique analytic solution of \eqref{eq:PINN:counterexample1} is given by $u(z) = a z + u_0$. For simplicity, we consider \eqref{eq:PINN:counterexample1} in an AD-PINN framework using ReLU-NNs restricted to the regularity class $\mcH_{\textrm{reg}}$.
Then one solves
\begin{equation}\label{eq:PINN:counterexample2}
\min_{u_\theta\in\mcH_{\textrm{reg}}} \frac{1}{N} \sum_{i=1}^N |u_\theta'(z_i) - a|^\nu + |u_\theta(0)-u_0|^\nu
\end{equation}
where $\{z_i\}_{i=1}^N\subset (0,T)$ are collocation points. Note that, since the solution $u$ of \eqref{eq:PINN:counterexample1} is an affine function, a ReLU-NN can exactly represent it. In particular, for any minimiser in this class $\mcH_{\textrm{reg}}$, the derivative $u_\theta'(z_i)$ is well-defined for all $i\in\{1,\ldots,N\}$.

Moreover, any minimiser $u_\theta\in\mcH_{\textrm{reg}}$ satisfies
\[
u_\theta(0)=u_0,
\qquad
u_\theta'(z_i)=a \quad \text{for all } i=1,\dots,N.
\]
In particular, the loss does not constrain the values $u_\theta(z_i)$ for all $i\in\{1,\ldots,N\}$ themselves, nor the behaviour of $u_\theta$ between the collocation points. Hence one can construct infinitely many minimisers of the form
\[
u_\theta(z) =
\begin{cases}
u_0 & \text{if } z = 0,\\
a z + \xi_i & \text{if } z = z_i,\ i\in\{1,\ldots,N\},\\
g(z) & \text{otherwise,}
\end{cases}
\]
where $\xi_i\in\R$ and $g:\R\to\R$ is a continuous and piecewise affine function, so that $u_\theta$ is continuous and piecewise affine.
Since each $\xi_i$ for $i\in\{1,\ldots,N\}$ and $g$ can be chosen arbitrarily within these constraints, there are infinitely many distinct minimisers.

If in \eqref{eq:PINN:counterexample2} the solution space $\mcH_{\textrm{reg}}$ is replaced by $\mcH^M_{\textrm{reg}}$ the problem persists, assuming that $M$ is sufficiently large. Assume $a>0$ and $u_0\ge 0$. Then $u_\theta(z) = \sigma(w_{1} z + u_0) + \sigma(w_{2} z - b)$ with $w_1,w_2 >0$, $w_1 + w_2 = a$, and $b \ge 0$ solves \eqref{eq:PINN:counterexample2} as long as $b < z_1 w_2$. In fact $u_\theta(0) = u_0$ and $u_\theta(z_i) = w_{1} z_i + u_0 + w_{2} z_i - b = a z_i + u_0 - b$ for $i\in\{1,\ldots,N\}$. Since this is a solution for any $b\in [0,z_1w_2)$, there are infinitely many minimizers and $M=7$ (1 hidden layer with 2 neurons) is already sufficiently large here.

For $M=4$ (1 hidden layer with a single neuron), different weight-bias configurations yield the same solution $u(z)=az+u_0$ for \nnew{$z\in[0,z_N]$, while possibly differing outside this range. Hence even in this setting the solution is not unique.} 
\nnew{In fact, $u_\theta(z) = -\sigma(- a z + b) + u_0+b$ is a solution of \eqref{eq:PINN:counterexample2} for any $b>a z_N$.
For $M=2$ (no hidden layer) one obtains a unique solution, namely $u_\theta(z) = wz +b$ with $w=a$ and $b=u_0$. Any other weight-bias configuration would yield a different function.}
\end{example}

This example illustrates a crucial issue of AD-PINNs. 
Namely, formulating the original differential equation \eqref{PhysicalModel} as an optimization problem in the form \eqref{eq:discretePINN:plain}  may render the solution not unique, even if the original problem \eqref{PhysicalModel} possesses exactly one solution in $\mcH$. 
This non-uniqueness originates from the fact that AD-PINNs only enforce the respective PDE in a finite number of collocation points allowing the solution to be arbitrary elsewhere, as illustrated in \cref{Example:PINN}. 
In particular this behavior may lead to the problem having an infinite number of solutions. It is then unclear which of these solutions is found by an optimization algorithm and it seems difficult to guarantee that the desired solution is found.

We are aware that for AD-PINNs, $\tanh$ is typically used as the activation function, since it is infinitely differentiable and therefore allows one to represent solutions that possess higher-order derivatives, as is the case for higher-order partial differential equations.
 However, the issue preserves and examples similar to \cref{Example:PINN} can be constructed. In fact, motivated by the above example we have the following general non-uniqueness result for AD-PINNs.

\begin{theorem}[Non-uniqueness of AD-PINN minimizers]\label{thm:nonuniq}
Fix finite collocation sets $\Omega^h=\{z_\mcF^i\}_{i=1}^{N_\mcF} \subset \Omega$ and $\Gamma^h=\{z_\mcB^j\}_{j=1}^{N_\mcB} \subset \Gamma$. 
Let $\mcH\subset U^h$ be a class of depth-$L$ neural networks satisfying
one of the following:
\begin{enumerate}[(i)]
\item \text{ReLU-NNs:} $\mcH = \mcH_{\mathrm{reg}}$ consists of depth-$L$ neural networks with ReLU activation functions and $L \ge \lceil \log_2(d+1)\rceil +1$.

\item \text{Smooth-activation neural networks:} 
The activation function $\sigma\in C^{\ell}(\R,\R)$ with 
\(
\ell := N_{\mcF}(r_{\mcF}+1)
      + N_{\mcB}(r_{\mcB}+1),
\)
satisfies $\sigma^{(k)}(a)\neq 0$ for some $a\in\R$ and all 
$0\le k \le \ell$, and the neural network depth satisfies $L\ge 2$.  
If $L>2$, we additionally assume that $\sigma$ is strictly monotone.
\end{enumerate}
If $\argmin_{u\in \mcH} \mcJ^h(u)$ has a solution, then it has infinitely many solutions in $\mcH$.
\end{theorem}
\begin{proof}
We start by showing that there exists a $\Phi\in\mcH$ such that 
\begin{equation}\label{eq:NullInterpolation}
\begin{split}
D^\beta \Phi(z_\mcF^i)=0 \qquad \text{for all } i=1,\ldots,N_\mcF, \ |\beta|\le r_{\mcF},\\
D^\gamma \Phi(z_\mcB^j)=0 \qquad \text{for all } j=1,\ldots,N_\mcB, \ |\gamma|\le r_{\mcB}.
\end{split}
\end{equation}
\begin{enumerate}[(i)]
\item Let $\mcH=\mcH_{\mathrm{reg}}$ be a class of ReLU-NNs. 
Then applying Lemma~\ref{Lem:ReLU:Null:interpolation} with $v\in\R^{c}\setminus\{0\}$ and $z_0\in {\Omega}\setminus\left(\Omega^h \cup \Gamma^h\right)$ yields the existence of a $\Phi\in\mcH$ of depth $L \geq \lceil \log_2(d+1)  \rceil + 1$ with $\Phi\equiv 0$ on an open neighborhood of each $z_\mcF^i$ and $z_\mcB^j$, and $\Phi\not\equiv 0$ on $\Omega\cup\Gamma$. In fact $\Phi(z_0)=v$. 
Hence all classical derivatives at $z_\mcF^i$ and all boundary traces at $z_\mcB^j$ vanish (in fact on neighborhoods). 

\item Let $\mcH$ be a class of neural networks with activation functions $\sigma\in C^{\ell}(\R, \R)$, where $\ell:= N_\mcF (r_\mcF+1) + N_\mcB (r_\mcB+1)$, and $\sigma^{(k)}(a)\neq 0$ for $0\le k\le \ell$ and some $a\in\R$. 
If $L>2$ then $\sigma$ is also strictly monotone.
Choose a $v\in\R^{c}\setminus\{0\}$, $z_0\in {\Omega}\setminus\left(\Omega^h \cup \Gamma^h\right)$ and a $v_*\in\R^d$ such that the projection $v_*\!\cdot z_\mcF^i$, $v_*\!\cdot z_\mcB^j$, and $v_*\!\cdot z_0$ are pairwise distinct for $i=1,\ldots,N_\mcF$, $j=1, \ldots,N_\mcB$ (see \cref{rem:projection-direction}). 
Then \cref{Lem:Null:Interpolation} yields the existence of a depth-$L$ neural network $\Phi\in\mcH$ with $L\geq 2$ such that $\Phi\not\equiv 0$ on $\Omega\cup \Gamma$ and has the desired properties \eqref{eq:NullInterpolation}. 
\end{enumerate}
Let $\hat{u}\in\mcH$ be a solution of $\argmin_{u\in \mcH} \mcJ^h(u)$. Then for any $\lambda\in\R$, $(\hat{u} + \lambda \Phi)(z) - u^*(z) = \hat{u}(z) - u^*(z)$ for all $z\in \mcD^h$ and by locality
$\mcF(\hat{u} + \lambda \Phi)(\nnew{z}) = \mcF(\hat{u})(\nnew{z})$ for all $\nnew{z}\in \Omega^h$  and $\mcB(\hat{u} + \lambda \Phi)(\nnew{z}) = \mcB(\hat{u})(\nnew{z})$ for all $\nnew{z}\in\Gamma^h$. This yields $\mcJ^h(\hat{u} + \lambda \Phi) = \mcJ^h(\hat{u})$ for all $\lambda\in\R$. 
Since $\hat{u},\Phi\in\mcH$ and $\mcH$ is closed under finite linear combinations, see \cref{sec:closure}, we obtain that $\hat{u} + \lambda \Phi \in \mcH$ for any $\lambda\in\R$ yielding infinitely many solutions of $\argmin_{u\in \mcH} \mcJ^h(u)$ in $\mcH$. 
\end{proof}

\begin{remark}
\begin{enumerate}[(a)]
\item We emphasize that the classes $\mcH$ in \cref{thm:nonuniq} do not set any width limitations on a neural network. This is essential as it yields the closure of $\mcH$ and $\mcH_{\mathrm{reg}}$ under finite linear combinations and allows to construct a neural network $\Phi\in\mcH$ with the desired interpolation properties \nnew{such that, for any $u\in\mcH$ and $\lambda\in\R$, the perturbed network $u+\lambda\Phi$ again belongs to $\mcH$}. Of course, the resulting neural network is of finite width. 
Hence the result persists for the class $\mcH^M$, if $M$ is sufficiently large. In practice, one usually chooses a large $M$ such that the approximation capabilities of the class $\mcH^M$ are high. However, if $M$ would be small, then \cref{thm:nonuniq} could break as we see in \cref{Example:PINN}, e.g., when \nnew{$M=2$ leading to a no-hidden-layer neural network}.

\item 
In the proof, utilizing \cref{Lem:ReLU:Null:interpolation,Lem:Null:Interpolation}, we enforced full Hermite interpolation conditions (all derivatives up to a certain order) on $\Omega^h\cup \Gamma^h$, which is stronger than necessary. It suffices to impose conditions only on the derivative orders at the respective points that actually appear in $\mcJ^h$ (including order 0).
\item Note that \cref{thm:nonuniq} holds for any values $\alpha_\mcF, \alpha_\mcB, \alpha_\mcD, \omega_\mcF^z, \omega_\mcB^z, \omega_\mcD^z \in \R$ as long as $\mcJ^h$ has a minimizer in $\mcH$.
\item In AD-PINNs, the activation functions are typically chosen to be smooth, nonlinear, and sufficiently differentiable, since the governing PDEs may involve higher-order derivatives. Nevertheless, in \cref{thm:nonuniq} we also consider ReLU activation functions, as they can be a reasonable choice for first-order PDEs; see, e.g., \cref{Example:PINN}. 
\end{enumerate}
\end{remark}

\begin{remark}\label{rem:overfitting}
\nnew{Assume $u_\theta \in \argmin_{u\in \mcH} \mcJ^h(u)$.} In the proof of \cref{thm:nonuniq} we constructed a nontrivial function $\Phi\in\mcH$, vanishing (together with all derivatives required by the PDE and boundary operators) at all interior and boundary collocation points, such that $u_\theta + \lambda \Phi \in \mcH$ is a minimizer of $\mcJ^h$ over $\mcH$ for every $\lambda\in\R$. 
In particular, the set of minimizers contains an unbounded affine line $\{u_\theta + \lambda\Phi : \lambda\in\R\}$, illustrating the severe non-uniqueness of the problem.
Let $\hat{u}\in U$ be a solution of the continuous PDE \eqref{PhysicalModel}. 
For any $\nu\in[1,\infty)$, the triangle inequality yields
\[
\|u_\theta + \lambda\Phi - \hat{u}\|_{L^\nu(\Omega)}
\;\ge\;
|\lambda|\|\Phi\|_{L^\nu(\Omega)}
\;-\;
\|u_\theta - \hat{u}\|_{L^\nu(\Omega)}
\longrightarrow \infty 
\qquad\text{as }\nnew{|}\lambda\nnew{|}\to\infty. 
\]
Thus, minimizers of the AD-PINN loss can diverge arbitrarily far from a true PDE solution. 
This structural non-uniqueness provides a mechanism that is closely related to the  ``overfitting'' effect observed in \cite{DoBiBo:23} for the heat equation, but it 
holds for general AD-PINN formulations. 
\end{remark}

\Cref{thm:nonuniq} shows that the AD-PINN problem is indeed ill-posed, since \eqref{eq:discretePINN:plain} admits infinitely many distinct minimizers. 
It does not, however, cover the optimization problems \eqref{eq:discretePINN:c} and \eqref{eq:discretePINN:rigid}. As discussed earlier, analogous results for the constrained and regularized formulations are more delicate, since the bound $c_\theta$ and the penalty term $\alpha_\theta |\theta|_q$ directly affect the weights and biases of the minimizers.
In particular, it does not seem obvious whether the constraint or the penalization could restore uniqueness, and a rigorous analysis of this question would likely require techniques beyond the scope of the present work.

\section{FD-PINN Framework}\label{Sec:FDPINN}

Instead of directly using \eqref{PhysicalModel} in an optimization framework, which leads to \eqref{eq:discretePINN:plain}, \eqref{eq:discretePINN:c} or \eqref{eq:discretePINN:rigid}, one may instead first discretize \eqref{PhysicalModel} by finite differences and subsequently apply the PINN methodology. A discrete version of \eqref{PhysicalModel} writes as
\begin{equation}\label{PhysicalModel:Discrete}
\begin{split}
&\mcF_h(u(z)) = 0 \qquad z\in \Omega^h,\\
&\mcB_h(u(z)) = 0 \qquad z  \in \Gamma^h,
\end{split}
\end{equation}
where $\mcF_h$, $\mcB_h$, $\Omega^h \subset \Omega$ and $\Gamma^h \subset \Gamma$ are finite difference discretizations of $\mcF$, $\mcB$, $\Omega$ and $\Gamma$, respectively, such that $\Omega^h\cap \Gamma^h = \emptyset$. 
In the finite difference setting, the unknown $u \in \R^{N\times c}$ represents the discrete function values at the $N$ grid points of the stencil $\Omega^h \cup \Gamma^h$. Accordingly, $u(z)\in\R^c$ denotes the value of $u$ at the grid point $z$, that is, the components of $u$ corresponding to the spatial node $z$.

If some measurement data $u^*\in \mcD^h \subseteq \Omega^h \cup \Gamma^h$ are given, then one may consider the following optimization problem 
\begin{equation}\label{eq:FDM:opt}
\begin{split}
\min_{u\in \R^{N\times c}}\; \Bigg\{\mcJ_{\FD}(u):= \alpha_{\mcF}\sum_{z\in\Omega^h} \omega_{\mcF}^z |\mcF_h(u(z))|_\nu^\nu 
&+ \alpha_{\mcB}\sum_{z \in \Gamma^h} \omega_{\mcB}^z |\mcB_h(u(z))|_\nu^\nu \\
&+ \alpha_\mcD \sum_{z\in\mcD^h} \omega_{\mcD}^z |u(z) - u^*(z)|_\nu^\nu\Bigg\},
\end{split}
\end{equation}
where $\alpha_{\mcF},\alpha_{\mcB},\alpha_{\mcD}\geq 0$ and $\omega_{\mcF}^z,\omega_{\mcB}^z,\omega_{\mcD}^z$ are suitable quadrature weights,
to find an approximate solution of \eqref{PhysicalModel:Discrete}. 
Applying the PINN methodology on \eqref{PhysicalModel:Discrete} yields
\begin{align}\label{eq:fulldiscretePINN}
\begin{split}
\min_{u_\theta\in \mcH_{c_\theta}^M} \Bigg\{ \mcJ_{\theta}(u_\theta):= \alpha_{\mcF}\sum_{z\in\Omega^h} \omega_{\mcF}^z |\mcF_h(u_\theta(z))|_\nu^\nu 
&+ \alpha_{\mcB}\sum_{z \in \Gamma^h} \omega_{\mcB}^z |\mcB_h(u_\theta(z))|_\nu^\nu \\
&+ \alpha_\mcD \sum_{z\in\mcD^h} \omega_{\mcD}^z |u_\theta(z) - u^*(z)|_\nu^\nu \Bigg\},
\end{split}
\end{align}
which is called FD-PINN.

\subsection{Existence of Minimizers}

It is well-known that if $\mcJ_\FD$ and $\mcJ_\theta$ is lower semicontinuous and coercive then \eqref{eq:FDM:opt} and \eqref{eq:fulldiscretePINN} attain its minimum. 
In particular, thanks to the Weierstraß theorem, \eqref{eq:fulldiscretePINN} has a solution if $c_\theta, M <\infty$, rendering $\mcH_{c_\theta}^M$ compact, and if $\mcJ_\theta$ is lower semicontinuous.
Moreover, we have the following obvious results.
\begin{proposition}\label{prop:Equivalence:Solution:FD}
\begin{enumerate}[(i)]
\item If $u_h\in\R^{N\times c}$ is a solution of \eqref{PhysicalModel:Discrete}, then it also solves \eqref{eq:FDM:opt} with $\mcJ_{\FD}(u_h)=0$ provided that $\alpha_D = 0$ or $\alpha_D > 0$ and $u_h(z) = u^*(z)$ for all $z\in\mcD^h$.
\item 
If $u_h\in\arg\min_{u\in \R^{N\times c}}\mcJ_{\FD}(u)$ with $\mcJ_{\FD}(u_h)=0$, then $u_h \in\R^{N\times c}$ solves \eqref{PhysicalModel:Discrete}.
\end{enumerate}
\end{proposition}
\begin{proof}
The statements follow directly by noting that if $\mcJ_{\FD}(u_h) = 0$, then we have $\mcF_h(u_h(z))=0$ for all $z\in\Omega^h$ and $\mcB_h(u_h(z))=0$ for all $z\in\Gamma^h$ and conversely. 
\end{proof}

Thanks to \cite[Theorem 5.1]{Pinkus:95} we know that if $\sigma\in C(\R,\R)$ is a non-polynomial activation function, then for any finite set of distinct input points \(\{z_i\}_{i=1}^N\subset \R^d\) and corresponding target values  \(\{\zeta_i\}_{i=1}^N \subset \R^c \) with $N\in\N$, there exists a one-hidden-layer neural network $\Phi:\R^d\to\R^c$ with $cN$ hidden neurons which interpolates this data, i.e., such that $\Phi(z_i)=\zeta_i$ for all $i=1,\ldots,N$. Based on this result we are able to prove \cite[Theorem 4.9]{LangerBehnamian:24} in our setting.
\begin{proposition}[{\cite[Theorem 4.9]{LangerBehnamian:24}}]\label{prop:minimizer:equivalence}
Consider finite collocation sets $\Omega^h \subset{\Omega}$ and $\Gamma^h\subset \Gamma$ with $\Omega^h \cap \Gamma^h=\emptyset$ and set $N=|\Omega^h \cup \Gamma^h|$. 
Let $c_\theta=M=\infty$ and $\mcH$ be a set of depth-$L$ neural networks with either
\begin{enumerate}[(i)]
\item ReLU activation functions and $L\geq 2$, or
\item non-polynomial activation functions $\sigma\in C(\R,\R)$ and $L=2$, or 
\item activation $\sigma \in C^{\ell-d}(\R,\R)$, $\ell:= N d$, satisfying $\sigma^{(k)}(a) \neq 0$ for some $a\in\R$ and all $0 \le k \le \ell-d$, and $L=2$.
\end{enumerate}
Then we have that
\begin{enumerate}[(a)]
\item if $u_h \in \R^{N\times c}$ is a solution of \eqref{eq:FDM:opt}, then there exist $u_\theta \in \mcH$ minimizing \eqref{eq:fulldiscretePINN} with $\mcJ_{\FD} (u_{h}) = \mcJ_{\theta}(u_{\theta})$ and $u_h(z) = u_\theta(z)$ for all $z\in\Omega^h \cup \Gamma^h$, and 
\item if $u_\theta \in \mcH$ is a solution of \eqref{eq:fulldiscretePINN}, then there exist $u_h\in\R^{N\times c}$ minimizing \eqref{eq:FDM:opt} with $\mcJ_{\FD} (u_{h}) = \mcJ_{\theta}(u_{\theta})$ and $u_h(z) = u_\theta(z)$ for all $z\in\Omega^h \cup \Gamma^h$.
\end{enumerate}
\end{proposition}
\begin{proof}
\begin{enumerate}[(a)]
\item 
Let $u_{h} \in \R^{N\times c}$ be any minimizer of $\mcJ_{\FD}$. By \cite[Theorem 5.1]{Pinkus:95}, for (i) and (ii), and by \cref{cor:pointwise-bc-vector}  (with $r_\mcF=0=r_\mcB$), for (iii), there exists a one-hidden-layer neural network (i.e., $L=2$) $u_{\theta}$ such that $u_{h}(z) =  u_{\theta}(z)$ for all $z\in\Omega^h\cup \Gamma^h$. 
In the case of ReLU activations, to obtain a depth-$L$ neural network with $L>2$ we just insert identity layers, cf., \cref{Sec:NN}, that do not change the values for all $z\in\Omega^h\cup \Gamma^h$. For simplicity we call this neural network again $ u_{\theta}$.
Then we have that $u_{h}(z) = {u}_{\theta}(z)$ for all $z\in\Omega^h\cup \Gamma^h$ and $\mcJ_{\FD} (u_{h}) = \mcJ_{\theta}(u_{\theta})$. 
To show that $u_\theta$ is optimal, we assume that there exists a $\tilde{u}_\theta\in\mcH$ with $\tilde{u}_{\theta} \neq u_{\theta}$ such that $\mcJ_{\theta}(\tilde{u}_{\theta})< \mcJ_{\theta}(u_{\theta})$. 
Then we can define $\tilde{u}_{h}\in\R^{N\times c}$ such that $\tilde{u}_{h}(z) = \tilde{u}_{\theta}(z)$ for all $z\in\Omega^h\cup \Gamma^h$. This yields $  \mcJ_{\FD} (\tilde{u}_{h})= \mcJ_{\theta}(\tilde{u}_{\theta})< \mcJ_{\theta}(u_{\theta}) = \mcJ_{\FD} (u_{h})$, which is a contradiction to the optimality of $u_{h}$.

\item 
Let $u_\theta\in\mcH$ be a minimizer of $\mcJ_\theta$. We define $u_{h}\in\R^{N\times c}$ such that $u_{h}(z) = u_\theta(z)$ for all $z\in\Omega^h\cup\Gamma^h$. 
This implies that $\mcJ_\theta (u_\theta) = \mcJ_{\FD}(u_{h}) $. 
Assume that $u_{h}$ is not a minimizer of $\mcJ_{\FD}$, i.e., there is a $\tilde{u}_{h}\in\R^{N\times c}$ with $\tilde{u}_{h}\neq {u}_{h}$ such that $\mcJ_{\FD}(\tilde{u}_{h}) < \mcJ_{\FD}(u_{h})$. 
By the same arguments as above, we construct an interpolation $\tilde{u}_\theta\in \mcH$ of $\tilde{u}_{h}$ such that $\tilde{u}_{h} (z) = \tilde{u}_\theta(z)$ for all $z\in\Omega^h\cup\Gamma^h$. 
Consequently $\mcJ_\theta (\tilde{u}_\theta) = \mcJ_{\FD}(\tilde{u}_{h}) < \mcJ_{\FD}(u_{h}) = \mcJ_\theta (u_\theta)$, which is a contradiction to $u_{\theta}$ being a minimizer of $\mcJ_\theta$ and hence $u_{h}$ is indeed a minimizer of $\mcJ_{\FD}$.
\end{enumerate}
\end{proof}

We emphasize that \cref{prop:minimizer:equivalence} ensures that, whenever a minimizer exists, the discrete finite-difference formulation \eqref{eq:FDM:opt} and the FD-PINN formulation \eqref{eq:fulldiscretePINN} admit minimizers that agree pointwise on the stencil $\Omega^h \cup \Gamma^h$. 
This equivalence plays a central role in our analysis below: it allows us to transfer existence and non-uniqueness properties between the two discrete formulations and to interpret FD-PINNs as neural network parameterizations of classical finite-difference schemes. Some further remarks on \cref{prop:minimizer:equivalence} are in order.

\begin{remark}
\begin{enumerate}[(a)]
\item Since $\mcJ_\theta$ considers only the values of a neural network at the collocation points and not their derivatives, no regularity needs to be requested for the used neural networks. Hence the class of ReLU-NNs does not need to be restricted to $\mcH_{\mathrm{reg}}$ in \cref{prop:minimizer:equivalence}.

\item 
In contrast to \cref{thm:nonuniq}, the ReLU-NNs used in \cref{prop:minimizer:equivalence} may have arbitrary depth $L\geq 2$. 
This is because \cref{prop:minimizer:equivalence} requires only a pointwise interpolation neural network, which can always be realized by a shallow ReLU architecture.
By comparison, the construction in the proof of \cref{thm:nonuniq} requires a ReLU-NN that is identically zero on nontrivial open sets while also taking prescribed values at selected points.
Implementing such a function with ReLU-NNs relies on the construction developed in the proof of \cref{Lem:ReLU:Null:interpolation}, which can be realized by a depth-$L$
neural network satisfying $L\geq \lceil \log_2(d+1)  \rceil + 1$. 
Thus, the depth restriction in \cref{thm:nonuniq} is not an inherent limitation of ReLU-NNs, but simply a consequence of the specific ``zero on open sets'' construction used in that proof.

\item Dimensions of the neural networks for which \cref{prop:minimizer:equivalence} holds:
\begin{enumerate}[(i)]
\item ReLU activation: $d_0=d$, $d_1=N$, $d_i=2$ for $i=2,\ldots,L-1$, $d_L=c$;
\item Continuous and non-polynomial activation: $d_0=d$, $d_1=N$, $d_L=c$;
\item $C^{\ell-d}$ activation: $d_0=d$, $d_1=c\binom{\nnew{\ell}}{d}$, $d_L=c$.
\end{enumerate}
Hence the result of \cref{prop:minimizer:equivalence} holds also for $M<\infty$, chosen according to these dimensions.
\end{enumerate}

\end{remark}

While \eqref{eq:fulldiscretePINN} with $c_\theta=M=\infty$ does not have a solution in general, by \cref{prop:minimizer:equivalence} it has one if \eqref{eq:FDM:opt} attains its minimum. 

The schematic, shown in \cref{fig:minimizer-structure:FD}, summarizes the logical relations among the three objects at the core of this section: the discrete PDE \eqref{PhysicalModel:Discrete}, the finite-difference functional $\mcJ_{\FD}$, and the FD-PINN objective $\mcJ_\theta$.
The diagram visualizes exactly the implications proved in
\cref{prop:Equivalence:Solution:FD,prop:minimizer:equivalence} for the regime $\alpha_{\mcD}=0$ and zero-loss solutions, highlighting how discrete PDE solutions correspond to minimizers of both optimization problems.

\begin{figure}
\begin{center}
\begin{tikzpicture}[
    >=stealth,
    node distance=2.5cm,
    formula/.style={draw, rounded corners, inner sep=4pt}
]

% Nodes
\node[formula, align=center] (f0) at (0, 3.5)  {\footnotesize{$u_h\in \R^{N\times c}$ solves \eqref{PhysicalModel:Discrete}}}; % top left
\node[formula, align=center] (f1) at (0, 1.5)  {\footnotesize{$u_h\in \argmin\limits_{u\in \R^{N\times c}} \mcJ_{\FD}(u)$}\\[1pt] \footnotesize{with $\alpha_\mcD=0$} \\[-2pt] \footnotesize{and $\mcJ_{\FD}(u_h)=0$}}; % top left
\node[formula, align=center] (f2) at (8, 1.5)  {\footnotesize{$u_\theta \in \argmin\limits_{u\in \mcH} \mcJ_\theta(u)$}\\[1pt] \footnotesize{with $\alpha_\mcD=0$ }\\[-2pt] \footnotesize{and $\mcJ_{\theta}(u_\theta)=0$}}; % top right

% Horizontal connections
\draw[->,thick,shorten >=2pt,shorten <=2pt]
  ([yshift=-3pt]f1.east)
    -- node[midway, below=-1pt] {\footnotesize{$\exists u_\theta \in \mcH$ (\cref{prop:minimizer:equivalence}~(i))}}
  ([yshift=-3pt]f2.west);
\draw[<-, thick,shorten >=2pt,shorten <=2pt] ([yshift=3pt]f1.east) -- node[midway, above=-1pt]{\footnotesize{$\exists u_h \in \R^{N\times c}$ (\cref{prop:minimizer:equivalence}~(ii))}} ([yshift=3pt]f2.west);

% Vertical connections
\draw[implies-implies,double equal sign distance, thick,shorten >=2pt,shorten <=2pt] (f0) -- (f1);
\end{tikzpicture}
\end{center}
\caption{Schematic overview of the two minimization problems for the discrete functionals $\mcJ_{\FD}$ and $\mcJ_\theta$ with the relation to the discrete PDE. Double arrows ($\Rightarrow$) indicate logical implications between the statements in the boxes established in \cref{prop:Equivalence:Solution:FD}. Single arrows ($\to$) represent the existence-type relations asserted in \cref{prop:minimizer:equivalence}, where a minimizer in one setting guarantees the existence of a corresponding minimizer in the other.}
\label{fig:minimizer-structure:FD}

\end{figure}
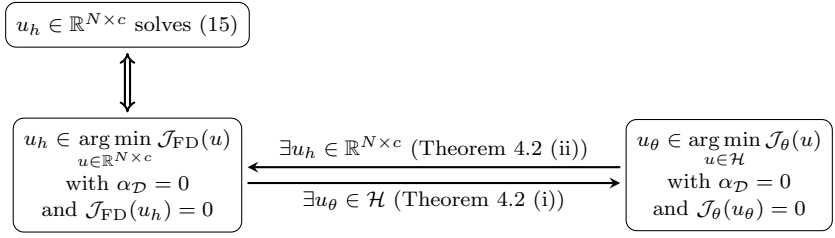

\subsection{Non-Uniqueness of Minimizers}

As in the AD-PINN formulation, the FD-PINN problem inherits the same ill-posedness: minimizers of the FD-PINN loss are never unique. This is made precise in the following theorem.

\begin{theorem}[Non-uniqueness of FD-PINN minimizers]\label{thm:nonuniq:FDPINN}
Consider finite collocation sets $\Omega^h \subset{\Omega}$ and $\Gamma^h\subset \Gamma$.
Let $L\ge 2$ and let $\mcH$ denote the set of depth-$L$ neural networks with either
\begin{enumerate}[(i)]
\item ReLU activation functions, or
\item non-polynomial activation functions $\sigma\in C(\R)$ that are strictly monotone if $L>2$, or
\item activation $\sigma \in C^{\ell}(\R,\R)$, $\ell:= |\Omega^h \cup \Gamma^h|$, satisfying $\sigma^{(k)}(a) \neq 0$ for some $a\in\R$ and all $0 \le k \le \ell$ that are strictly monotone if $L>2$.
\end{enumerate}
If $\argmin_{u_\theta \in \mcH} \mcJ_{\theta}(u_\theta)$ has a solution, then it has infinitely many solutions in $\mcH$.
\end{theorem}
\begin{proof}
The proof follows the same idea as the proof of \cref{thm:nonuniq}. However, in this context it suffices to construct a neural network $\Phi\in\mcH$, not identically zero, that satisfies the interpolation conditions $\Phi(z) = 0$ for all $z\in\Omega^h\cup \Gamma^h$.

By \cite[Theorem 5.1]{Pinkus:95}, for (i) and (ii), and by \cref{Lem:Null:Interpolation} (with $r_\mcF=0=r_\mcB$ and suitable $v_*\in\R^d$), for (iii), there exists a one-hidden-layer neural network $\tilde{\Phi}$ with this property, i.e., $\tilde{\Phi}(z)=0$ for $z\in\Omega^h\cup \Gamma^h$ and  $\tilde{\Phi}(z_0) \not=0$ for some $z_0 \in \Omega\setminus\{\Omega^h\cup \Gamma^h\}$. To obtain a depth-$L$ neural network we just insert layers that do not change the interpolation conditions and keep a non-zero value in $z_0$. This can be realized as in the proof of \cref{Lem:Null:Interpolation}, for (ii) and (iii), due to the strict monotonicity of $\sigma$, and as in the proof of \cref{Lem:ReLU:Null:interpolation} for (i) by adding identity layers (see also \cref{Sec:NN}), yielding $\Phi\in\mcH$ such that ${\Phi}(z)=0$ for $z\in\Omega^h\cup \Gamma^h$ and  ${\Phi}(z_0) \not=0$.

Let $\hat{u}\in\mcH$ be a solution of $\argmin_{u\in \mcH} \mcJ_{\theta}(u)$. Then for any $\lambda\in\R$, $(\hat{u} + \lambda \Phi)(z) = \hat{u}(z)$ for all $z\in \Omega^h \cup \Gamma^h$ and hence $\mcJ_{\theta}(\hat{u} + \lambda \Phi) = \mcJ_{\theta}(\hat{u})$. 
Since $\hat{u},\Phi\in\mcH$ and $\mcH$ is closed under finite linear combinations, see \cref{sec:closure}, we obtain that $\hat{u} + \lambda \Phi \in \mcH$ for any $\lambda\in\R$ yielding infinitely many solutions of $\argmin_{u\in \mcH} \mcJ_{\theta}(u)$ in $\mcH$. 
\end{proof}
\begin{remark}\label{rem:overfitting:FDM}
The construction used in the proof of \cref{thm:nonuniq:FDPINN} is analogous to that in the proof of \cref{thm:nonuniq}: 
we again construct a nontrivial $\Phi\in\mcH$ that vanishes on $\Omega^h\cup\Gamma^h$ and such that $\hat u + \lambda\Phi\in\mcH$ is a minimizer of $\mcJ_\theta$ for every $\lambda\in\R$. 
Let $u_h\in\R^{N\times c}$ denote a solution of the discrete PDE \eqref{PhysicalModel:Discrete}. Then, in contrast to \cref{rem:overfitting} in the AD-PINN setting, we obtain 
\[
\sum_{z\in\Omega^h\cup\Gamma^h} 
|\hat u(z)+\lambda\Phi(z)-u_h(z)|
=
\sum_{z\in\Omega^h\cup\Gamma^h} 
|\hat u(z)-u_h(z)|
\qquad\text{for all }\lambda\in\R,
\]
and in particular this sum vanishes for all $\lambda$ if $\mcJ_\theta(\hat u)=0$, since by \cref{prop:minimizer:equivalence,prop:Equivalence:Solution:FD} the FD-PINN minimizer $\hat{u}$ then coincides with $u_h$ on $\Omega^h\cup\Gamma^h$.
Thus, while FD-PINNs exhibit the same affine non-uniqueness in $\mcH$ as 
AD-PINNs, this non-uniqueness does not alter the discrete finite-difference 
solution on the stencil.
\end{remark}

\paragraph{Implications of non-uniqueness: AD-PINNs vs.\ FD-PINNs}
The non-uniqueness results above show that both AD-PINNs and FD-PINNs admit infinitely many minimizers of the respective loss, so that the corresponding optimization problems are ill-posed. 
For FD-PINNs, however, the situation is substantially less problematic from the perspective of PDE approximation in the regime where a zero-loss solution exists; see also \cref{rem:overfitting:FDM}.
In this situation, which occurs whenever the discrete finite-difference problem \eqref{PhysicalModel:Discrete} admits a solution and $\alpha_{\mcD}=0$, \cref{prop:Equivalence:Solution:FD,prop:minimizer:equivalence} imply that every FD-PINN minimizer $u_\theta$ with $\mcJ_\theta(u_\theta)=0$ coincides with a finite-difference solution at all stencil points. 
Thus, in this specific zero-residual regime, while FD-PINNs are ill-posed as function-approximation problems -- infinitely many distinct continuous extensions exist -- they are effectively unique on the grid whenever the discrete PDE admits a unique solution.
If the discrete PDE is not uniquely solvable, then FD-PINN minimizers reproduce different discrete solutions accordingly and uniqueness on the grid is obviously not guaranteed in this case.

For AD-PINNs, the picture is less favorable, as they do not enjoy such grid-level uniqueness. If the differential operator $\mcF$ contains no zeroth-order term, the residual depends only on derivatives of $u$, making it possible for two distinct AD-PINN minimizers to disagree already on $\Omega^h \cup \Gamma^h$ while achieving the same loss value, cf., \cref{Example:PINN}.
 Even when zeroth-order terms or data-misfit terms are present, our analysis does not provide an analogue of the grid-level uniqueness enjoyed by FD-PINNs. 
Consequently, an AD-PINN minimizer may not be tied to any underlying consistent finite-difference scheme, and different minimizers may represent qualitatively different approximate solutions, even if they achieve identical loss values. 
In fact, as shown in \cref{rem:overfitting}, the AD-PINN minimizer set may 
contain functions that deviate arbitrarily far from the true solution of the 
continuous PDE while attaining the same loss value. 
\nnew{Nevertheless, uniqueness of the minimizer values at the collocation points for AD-PINNs can only be guaranteed under additional assumptions. For instance, if the discrete loss functional $\mcJ^h$ is strictly convex on $\mcH$ and a minimizer exists, then all minimizers agree on the collocation points. Indeed, suppose $u_1,u_2\in\mcH$ are two distinct minimizers of $\mcJ^h$. By strict convexity, $\mcJ^h(\frac{u_1+u_2}{2}) < \frac{1}{2}\mcJ^h(u_1) + \frac{1}{2}\mcJ^h(u_2)$ which contradicts the minimality of $u_1$ and $u_2$, since $\tfrac{u_1+u_2}{2}\in\mcH$ by closure of $\mcH$ under linear combinations.}

The structural reason for this discrepancy is that AD-PINNs compute derivatives by automatic differentiation pointwise, while FD-PINNs approximate derivatives through finite-difference stencils that couple neighboring nodes.
This local coupling prevents pointwise isolation, so FD-PINNs do not possess the pointwise freedom present in AD-PINNs. 
It is precisely this structural restriction that forces all zero-loss FD-PINN minimizers to agree on the stencil (whenever the discrete PDE solution is unique), even though they may differ between grid points.

This distinction also clarifies the conceptual diagram in \cref{fig:minimizer-structure,fig:minimizer-structure:FD}: in the AD-PINN setting (\cref{fig:minimizer-structure}) the flow of information runs only from the continuous PDE to the AD-PINN loss, whereas in the FD-PINN formulation (\cref{fig:minimizer-structure:FD}) there is a two-way correspondence between the discrete PDE and the FD-PINN objective. 
The above explained contrast and the bidirectional link explain why FD-PINNs can be interpreted as neural parameterizations of a classical finite-difference discretization, while the AD-PINN problem is ill-posed without a corresponding uniqueness guarantee at the collocation points.

\section{Numerical Experiments}\label{Sec:NumericalExperiments}

The following experiments are not intended to demonstrate algorithmic novelty but provide representative cases illustrating typical behaviors of AD-PINNs and FD-PINNs. 
The experiments confirm that theoretical ill-posedness translates into practical instability for AD-PINNs, while FD-PINNs seem to constrain the solution space more favorably.
 
Three examples are considered. For FD-PINNs we use ReLU activation functions throughout, reflecting the fact that, in this formulation, derivatives are computed via finite differences and no additional smoothness of the activation function is required. 
 The first example concerns a Poisson problem with nontrivial boundary conditions, where we demonstrate that AD-PINNs can fail to converge to the correct solution, while FD-PINNs succeed.
The second example addresses a time-dependent Schrödinger equation, serving as a representative forward problem where FD-PINNs perform comparably to AD-PINNs. 
To the best of our knowledge, FD-PINNs have not previously been evaluated on oscillatory Schrödinger-type problems using nonsmooth activations such as ReLU; this example therefore also illustrates that FD-PINNs remain effective without smooth activation functions.
The third example treats an inverse Navier-Stokes problem to demonstrate that FD-PINNs can also handle data-driven tasks similarly to AD-PINNs. To our knowledge, FD-PINNs have been less explored in data‐driven/inverse contexts, and our third numerical example addresses this gap.

Before turning to the individual examples, we summarize the general numerical setup used throughout this section.
All neural networks are implemented in Python using TensorFlow \cite{tensorflow2015} and are trained with the \nnew{TensorFlow's built-in} Adam optimizer \cite{KingmaBa:15} with a fixed learning rate of $10^{-3}$. 
We deliberately refrain from employing multi-stage optimization strategies such as Adam$\to$L-BFGS or from tuning network architectures for optimal accuracy, since the purpose of the experiments is to assess the behavior of the FD-PINN formulation under a consistent and standard setup rather than to optimize performance. 
Unless stated otherwise, the architectures used in the examples therefore provide sufficient expressive capacity but are not further tuned. 
At each iteration we evaluate the current objective and update the stored approximation $u_\theta$ only if the new iterate attains a strictly smaller loss than all previous ones. 
In this way, the sequence of recorded energies is monotonically decreasing and the final reported network corresponds to the best objective value observed along the optimization trajectory.
The implementation is publicly available at \url{https://github.com/andreastvlanger/PINN}.

\subsection{Poisson Equation with Singularity}
We consider the two dimensional Poisson equation with homogeneous boundary conditions given as
\begin{equation}\label{eq:Poisson}
\begin{split}
-\Delta u &= 1 \qquad \text{in } \Omega,\\
u&=0 \qquad \text{on } \Gamma = \partial \Omega,
\end{split}
\end{equation}
where $\Omega=(-1,1)^2 \setminus \left([0,1)\times\{0\}\right)\subset\R^2$ and $u:\Omega\to \R$, i.e., $d=2$ and $c=1$. This model problem follows the setup introduced in \cite{EYu:18}, where it was used to study the deep Ritz method. 

The domain $\Omega$ is uniformly discretized with mesh size $h=0.05$ in both dimensions, yielding the discrete domain $\Omega^h = \{(x_i,y_j) \mid x_i = -1 + hi,\ y_j=-1+hj, \ i,j\in\Z,\ -1\leq x_i,y_j\leq 1\}$ of collocation points. Let $N\in\N$ be the number of collocation points, i.e., here we have \nnew{$N_\mcF=1501$} points inside the domain and $N_\mcB = \nnew{180}$ boundary points, yielding $N = 41 \times 41 = 1\,681$ points in total. 
Utilizing these points, a finite difference method (FDM) of \eqref{eq:Poisson} yields \begin{equation}\label{Eq:Poisson:FDM}
A u^h = b,
\end{equation}
where $A\in\R^{N \times N}$ is a standard finite difference discretization of the Laplacian $\Delta$ with incorporated Dirichlet boundary conditions, $b\in\R^N$ is the respective discretized right hand side, and $u^h\in\R^N$ the associated finite difference solution depicted in \cref{Fig:Poisson:FDM}. 

\begin{figure}[h]
\begin{center}
\newcommand{\mysubfigwidth}{0.33\textwidth}
\newcommand{\figscale}{0.35}

\subfloat[Solution of FDM\label{Fig:Poisson:FDM}]{
    \includegraphics[scale=\figscale]{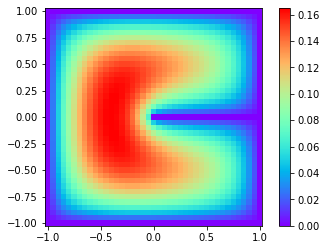}
  }\hfill
  \subfloat[Solution of FD-PINN \label{Fig:Poisson:FDPINN}]{
    \includegraphics[scale=\figscale]{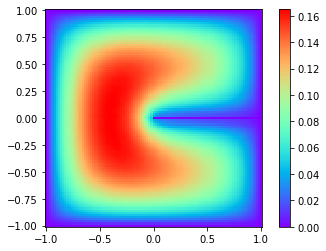}
  }\hfill
  \subfloat[Solution of AD-PINN with hard boundary conditions. \label{Fig:Poisson:PINN:bcTrue}]{
    \includegraphics[scale=\figscale]{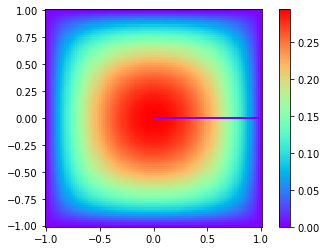}
  }\hfill
  \subfloat[Solution of AD-PINN with $\alpha_\mcB=1$ \label{Fig:Poisson:PINN:1}]{
  \includegraphics[scale=\figscale]{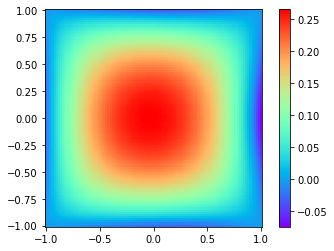}
  }\hfill
  \subfloat[Solution of AD-PINN with $\alpha_\mcB=100$ \label{Fig:Poisson:PINN:100}]{
  \includegraphics[scale=\figscale]{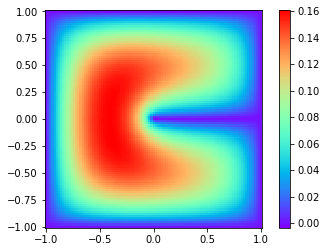}
  }\hfill
  \subfloat[Solution of AD-PINN with $\alpha_\mcB=10\,000$ \label{Fig:Poisson:PINN:10000}]{
  \includegraphics[scale=\figscale]{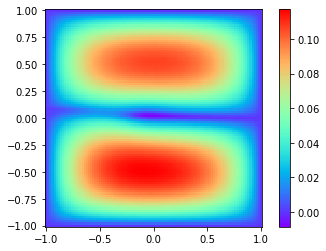}
  }
\end{center}
\caption{Solutions of the Poisson problem \eqref{eq:Poisson} obtained by FDM, FD-PINN and AD-PINN.\label{Fig:Poisson}}
\end{figure}

To obtain an FD-PINN solution, we utilize \eqref{Eq:Poisson:FDM}. Then \eqref{eq:fulldiscretePINN} can be written as
\begin{align*}
\min_{\theta \in \R^M}\; & \nnew{h^2}|A u_\theta^h - b|_2^2,
\end{align*}
where $u_\theta^h$ is the vector of the values of $u_\theta$ sampled at the grid points of $\Omega^h$. Here we set $\alpha_{\mcF} = 1$ and $\nu=2$.
Note that all boundary conditions are included in $A$, and hence the first and second terms in \eqref{eq:fulldiscretePINN} merge into one term. 
To enforce the boundary conditions even more strongly, we implement them into the neural network, i.e., we search for a solution $u_\theta$ in the space $\{u_\theta \in \mcH \colon u_\theta(z) = 0 \text{ for } z\in \Gamma\}$. 
Moreover, the neural network architecture is specified as follows: the neural network consists of an input layer with 2 neurons, 7 hidden layers each having 32 neurons and ReLU activation functions, and an output layer with 1 neuron.
The overall optimization (learning) process is run for 200\,000 iterations. The FD-PINN solution is shown in \cref{Fig:Poisson:FDPINN} with resolution $101 \times 101$. Note that we can depict the solution with a finer resolutions, as the solution is a continuous function.

An AD-PINN tackles \eqref{eq:Poisson} by solving
\begin{align}\label{eq:Poisson:PINN}
\min_{\theta \in \R^M}\; & \frac{1}{N_{\mcF}}\sum_{i=1}^{N_\mcF} |-\Delta u_\theta(z_\mcF^i)  - 1|^2 + \alpha_\mcB\frac{1}{N_{\mcB}} \sum_{i=1}^{N_\mcB} |u_\theta(z_\mcB^i)|^2, 
\end{align}
where $\Omega^h = \{z_\mcF^i\}_{i=1}^{N_\mcF}$ are the collocation points inside the domain and $\Gamma^h = \{z_\mcB^i\}_{i=1}^{N_\mcB}$ are the collocation points on the boundary. We search for a solution among the neural networks consisting of an input layer with 2 nodes, 7 hidden layers each with 32 nodes, an output layer with 1 node and $\tanh$ activation functions on all nodes in the hidden layers. 
We compute the solution under two settings:
\begin{enumerate*}[(i)]
\item\label{settingi} incorporating the boundary conditions directly into the neural network, and
\item\label{settingii} without incorporating the boundary conditions.
\end{enumerate*} 
In setting \ref{settingi} the choice of $\alpha_\mcB$ is irrelevant, which looks pleasant at first sight, as its choice is a priori not clear. In setting \ref{settingii} we consider $\alpha_\mcB \in \{1, 100, 10\,000\}$ to show the influence of the parameter on the solution process. Again the optimization process is terminated after 200\,000 iterations.
The respective obtained results are shown in \cref{Fig:Poisson:PINN:1,Fig:Poisson:PINN:100,Fig:Poisson:PINN:10000} at a resolution of $101 \times 101$.

\paragraph{Interpretation}
In the AD-PINN approach, the choice of the parameter $\alpha_\mcB$, or more generally the treatment of the boundary conditions, is delicate. 
In particular, \cref{Fig:Poisson} shows that when $\alpha_\mcB$ is either too small or too large, no suitable approximations is obtained within $200\,000$ iterations. After this many iterations, the loss remains around $0.0047$ for $\alpha_\mcB=1$ and $0.00038$ for $\alpha_\mcB=10\,000$, indicating that substantially more iterations would be required to reach a satisfactory solution. 
For $\alpha_\mcB=100$ a reasonable approximation is generated, yielding a loss of around $6.9656\cdot10^{-6}$. However, we see in \cref{Fig:Poisson:PINN:100} that the boundary conditions do not hold exactly, while this is the case for the solutions of the FDM, the FD-PINN, and the AD-PINN with hard boundary conditions. 
Interestingly, the AD-PINN with hard boundary conditions finds another solution of \eqref{eq:Poisson:PINN} than the AD-PINN with $\alpha_\mcB=100$. 
The loss evaluated at the solution in \cref{Fig:Poisson:PINN:bcTrue} yields approximately $5.8781 \cdot 10^{-8}$ indicating that it is indeed a very close approximation of a solution. 
Note that in all our computations we use single-precision floating-point format (IEEE-754 float32), which has a machine epsilon of around $1.19 \cdot 10^{-7}$.
As shown in \cref{thm:nonuniq}, the AD-PINN approach in general does not have a unique solution. 
Here, by incorporating the boundary conditions into the neural network, we are able to find an alternative solution numerically, i.e., a minimizer which is not a solution of the original PDE problem.

Let us elaborate why the function depicted in \cref{Fig:Poisson:PINN:bcTrue} is reasonable as a solution of \eqref{eq:Poisson:PINN}. As the boundary conditions always hold, the second term in \eqref{eq:Poisson:PINN} is always 0 and hence does not influence the optimization process. 
The continuous Laplacian is a local operator and in \eqref{eq:Poisson:PINN} is only evaluated at a finite number of points inside the domain. 
Hence the first term in \eqref{eq:Poisson:PINN} does not see the boundary, and the optimization procedure somehow overlooks the boundary at $[0,1) \times \{0\}$, as the boundary conditions already hold there anyway. 
Note that the behavior in a very close vicinity of the boundary, where no collocation point is, does not affect the energy at all. 
Hence, in this region the PDE is effectively unenforced, and the neural network output can vary freely, subject only to the structural constraints of the chosen neural network class, as noted in \cref{Example:PINN}.

In contrast, the linear system \eqref{Eq:Poisson:FDM} contains the boundary conditions, and hence the objective function of the FD-PINN approach is always influenced by the boundary conditions. Hence, in this sense, the FD-PINN seems to be superior to the AD-PINN.

\subsection{Schrödinger Equation}

We consider the time-dependent nonlinear Schrö\-dinger equation from \cite{RaPeKa:19} with periodic boundary conditions given as
\begin{equation}\label{eq:Schrodinger}
\begin{split}
&\im \frac{\partial\psi}{\partial t} + 0.5 \frac{\partial^2 \psi}{\partial x^2} + |\psi|^2 \psi = 0, \qquad x\in [-5,5], \ t\in [0,2\pi],\\
&\psi(0,x) = 2 \operatorname{sech}(x),\qquad x\in [-5,5] \\
&\psi(t,-5) = \psi(t,5), \ \frac{\partial \psi}{\partial x}(t,-5) = \frac{\partial \psi}{\partial x}(t,5),\qquad t\in [0,2\pi],
\end{split}
\end{equation}
where $\psi$ represents a complex valued function and $\im$ denotes the complex number $\sqrt{-1}$.

To build the FD-PINN, we need to discretize the PDE and the respective boundary conditions. In particular, we discretize time equidistantly into $T+1$ points, defined by $t_k = k h_t$ for $k=0,\ldots,T$, where $h_t=\frac{2 \pi}{T}$. Similarly, the spatial domain is divided into $N+1$ equidistant points given by $x_j = -5 + j h_x$, $j=0,\ldots,N$, with $h_x=\frac{10}{N}$.
Then, introducing the ghost point $x_{-1} := x_0 - h_x$, the periodic Neumann boundary condition is approximated by
\[
\frac{\partial\psi}{\partial x}(t_k,-5) \approx \frac{\psi(t_k,x_0) - \psi(t_k,x_{-1})}{h_x} = \frac{\psi(t_k,x_{N}) - \psi(t_k,x_{N-1})}{h_x} \approx \frac{\partial\psi}{\partial x}(t_k,5)
\]
for all $k=0,\ldots,T$.
This, together with the periodic Dirichlet boundary condition $\psi(t_k,x_0) = \psi(t_k,x_{N})$ yields $\psi(t_k,x_{-1}) = \psi(t_k,x_{N-1})$ for all $k=0,\ldots,T$. This allows us to incorporate the boundary conditions directly into the discretized PDE. 
Thereby, the PDE is approximated in time using the implicit Euler method and in space using standard finite difference schemes, leading to
\begin{equation*}
\begin{split}
f(t_k,x_j):=&\im \frac{\psi(t_{k+1},x_j)-\psi(t_{k},x_j)}{h_t} \\
&+ \frac{0.5\left(\psi(t_{k+1},x_{j+1}) - 2 \psi(t_{k+1},x_{j}) + \psi(t_{k+1},x_{j-1})\right)}{h_x^2}  \\
&+ |\psi(t_{k+1},x_{j})|^2 \psi(t_{k+1},x_{j})
\end{split}
\end{equation*}
for $k=0,\ldots,T-1$ and $j=0,\ldots,N-1$, where $x_{-1}:=x_{N-1}$.

 The loss function then reads as
\begin{equation}\label{eq:SchroedingerFDPINN}
 \frac{1}{N T}\sum_{j=0}^{N-1} \sum_{k=0}^{T-1} |f(t_k,x_j)|^2 + \frac{1}{N+1}\sum_{j=0}^N |\psi(0,x_j) - 2 \operatorname{sech}(x_j)|^2.
\end{equation}
The neural network is constructed such that it has 2 input neurons (time and space) and two output neurons, representing the real and imaginary part of $\psi$. 
It consists of 20 hidden layers, each having 100 neurons, which should give the neural network sufficient approximation capacity when using ReLU activation functions. 
The architecture is not optimized in any way, as our aim here is solely to evaluate the approximation capability of the FD-PINN formulation.
Further, we incorporate the initial value into the neural network directly, i.e., the solution is searched in the space $\left\{ \psi \in \mcH \colon \psi(0,x_j) = 2 \operatorname{sech}(x_j) \text{ for all } j\in\{0,\ldots,N\}\right\}$, yielding the second term in \eqref{eq:SchroedingerFDPINN} always equal to zero. In our numerical experiment we use $N=100$ and $T=500$ yielding 50\,000 data points equidistantly meshing the time-space domain. 
The overall optimization (learning) process is terminated after 350\,000 iterations. 
All initial-condition and reference data follow the setup of \cite{RaPeKa:19}, and were taken from \url{https://github.com/maziarraissi/PINNs}.

\begin{figure}[h]
\centering
\newcommand{\mysubfigwidth}{0.3\textwidth}

\subfloat[Magnitude of the FD-PINN predicted solution $|\psi|$
\label{Fig:Schrodinger:FDPINN:s}]{
    \includegraphics[width=8cm, height=4cm]{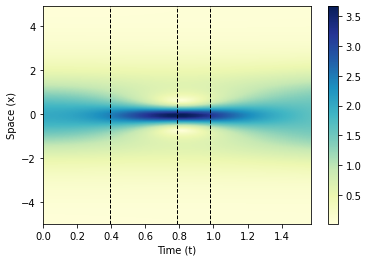}
}\\[1ex]
\subfloat[Comparison of exact and FD-PINN predicted solution at $t=0.393$
\label{Fig:Schrodinger:FDPINN:1}]{
    \includegraphics[width=\mysubfigwidth]{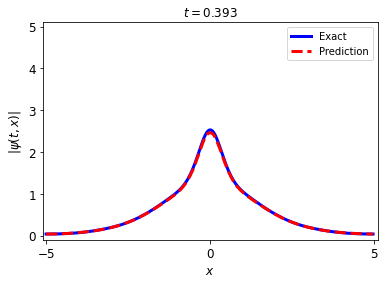}
}\hfill
\subfloat[Comparison of exact and FD-PINN predicted solution at $t=0.785$
\label{Fig:Schrodinger:FDPINN:2}]{
    \includegraphics[width=\mysubfigwidth]{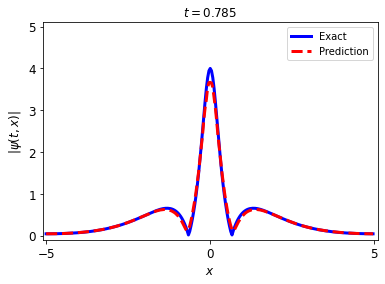}
}\hfill
\subfloat[Comparison of exact and FD-PINN predicted solution at $t=0.982$
\label{Fig:Schrodinger:FDPINN:3}]{
    \includegraphics[width=\mysubfigwidth]{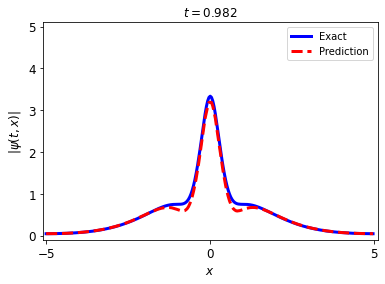}
}

\caption{Solutions of the Schrödinger equation \eqref{eq:Schrodinger} obtained by the FD-PINN.}
\label{Fig:Schrodinger:FDPINN}
\end{figure}

\begin{figure}[h]
\centering
\newcommand{\mysubfigwidth}{0.3\textwidth}

\subfloat[Magnitude of the AD-PINN predicted solution $|\psi|$
\label{Fig:Schrodinger:PINN:s}]{
    \includegraphics[width=8cm, height=4cm]{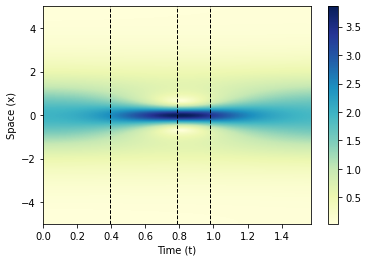}
}\\[1ex]

\subfloat[Comparison of exact and AD-PINN predicted solution at $t=0.393$
\label{Fig:Schrodinger:PINN:1}]{
    \includegraphics[width=\mysubfigwidth]{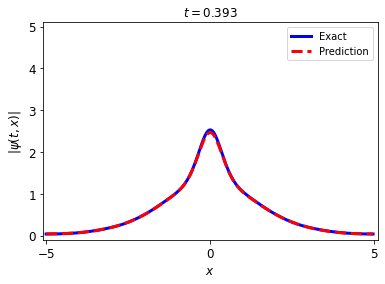}
}\hfill
\subfloat[Comparison of exact and AD-PINN predicted solution at $t=0.785$
\label{Fig:Schrodinger:PINN:2}]{
    \includegraphics[width=\mysubfigwidth]{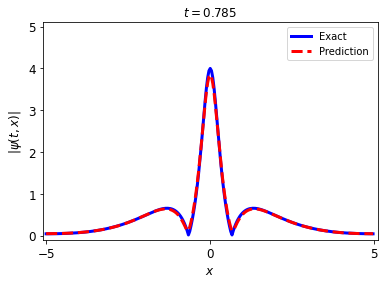}
}\hfill
\subfloat[Comparison of exact and AD-PINN predicted solution at $t=0.982$
\label{Fig:Schrodinger:PINN:3}]{
    \includegraphics[width=\mysubfigwidth]{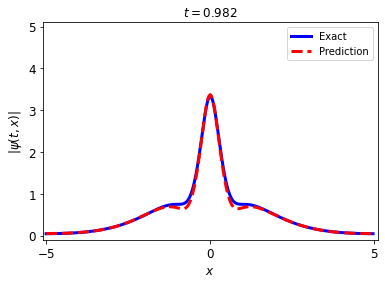}
}
\caption{Solutions of the Schrödinger equation \eqref{eq:Schrodinger} obtained by the AD-PINN.}
\label{Fig:Schrodinger:PINN}
\end{figure}

The solution obtained with the FD-PINN is illustrated in \cref{Fig:Schrodinger:FDPINN}. In \cref{Fig:Schrodinger:FDPINN:s}, we display the magnitude of the predicted solution $|\psi|$. 
The prediction accuracy, evaluated on the equidistant test mesh, yields a relative $L^2$-error of $6.4 \times 10^{-2}$. A more detailed assessment is provided in \cref{Fig:Schrodinger:FDPINN:1,Fig:Schrodinger:FDPINN:2,Fig:Schrodinger:FDPINN:3}, where the predicted solution is compared with the exact one at representative time instants $t = 0.393, 0.785, 0.982$. 
These comparisons show that, even when trained with only limited initial-condition data, the FD-PINN successfully captures the nonlinear dynamics of the Schrödinger equation.

For reference, we compare the FD-PINN against the AD-PINN applied to the same problem, cf., \cite{RaPeKa:19} and \cref{Fig:Schrodinger:PINN}. The AD-PINN consists of four hidden layers with 100 neurons each, employs $\tanh$ activation functions, and is trained for $20\,000$ iterations. 
This configuration yields a relative $L^2$-error of $3.3 \times 10^{-2}$, which is only slightly smaller than that of the FD-PINN, showing that both perform comparable.

\subsection{Data-driven Discovery of Partial Differential Equations (Navier-Stokes equation)}
We consider the two-dimensional incompressible Navier-Stokes equations in the velocity-pressure formulation:
\begin{align*}
\frac{\partial u}{\partial t} + \lambda_1 \left(u \frac{\partial u}{\partial x} + v \frac{\partial u}{\partial y}\right) &= -\frac{\partial p}{\partial x} + \lambda_2 \left(\frac{\partial^2 u}{\partial x^2} + \frac{\partial^2 u}{\partial y^2}\right),\\
\frac{\partial v}{\partial t} + \lambda_1 \left(u \frac{\partial v}{\partial x} + v \frac{\partial v}{\partial y}\right) &= -\frac{\partial p}{\partial y} + \lambda_2 \left(\frac{\partial^2 v}{\partial x^2} + \frac{\partial^2 v}{\partial y^2}\right),\\
\frac{\partial u}{\partial x} + \frac{\partial v}{\partial y} &= 0,
\end{align*}
where $(u,v)$ denote the velocity components, $p$ the pressure, and $\lambda_1$, $\lambda_2$ are coefficients corresponding to convection and viscosity. Our objective is to recover $\lambda_1$, $\lambda_2$ and the pressure field $p$ from velocity observations alone.

\paragraph{Data generation} 
Training data were generated by a finite difference solver for the two-dimensional
incompressible Navier-Stokes equations on a periodic square domain
$[0,2\pi)\times[0,2\pi)$, discretized with $32\times 32$ grid points, i.e., 
\[
\Omega^h = \left\{ (x_i,y_j) \;\middle|\; 
x_i = i h_x,\; y_j = j h_y,\;
i=0,\ldots,31,\; j=0,\ldots,31 \right\},\] 
with \(h_x = h_y = \tfrac{2\pi}{32}.\) Spatial derivatives
were approximated by second-order central finite differences, while time stepping was
performed with an implicit Euler discretization of the diffusive terms and an explicit
treatment of convection. The resulting nonlinear implicit system at each time step was
solved by a fixed-point iteration. 

Within each fixed-point iteration, a Helmholtz problem was solved for an intermediate velocity field using the current iterate of the nonlinear term. 
To enforce incompressibility, this velocity was projected onto the divergence-free subspace by solving a discrete Poisson equation for a scalar correction potential $\phi$ and updating
\[
u^{k+1} = \tilde{u} - h_t\, \frac{\partial \phi}{\partial x}, \qquad
v^{k+1} = \tilde{v} - h_t\, \frac{\partial \phi}{\partial y}, \qquad
p^{k+1} = p^k + \phi,
\]
where $h_t$ is the temporal step size of the finite difference solver, 
$k$ denotes the index of the fixed-point iteration within the current time step, 
$(\tilde{u},\tilde{v})$ is the intermediate velocity, and $p$ the pressure; see \cite{GuMiSh:06}. 
This projection method ensures that each fixed-point iterate satisfies the discrete 
divergence-free constraint, and convergence is declared once successive iterates differ by less than a prescribed tolerance. 

To prevent nonlinear advection from being absorbed into the pressure gradient, the initial condition was chosen as a multi-mode divergence-free streamfunction,
\begin{equation*}
\begin{split}
    \psi(x,y)~=~&1.00 \,\sin(x)\cos(y)
                 + 0.30 \,\sin(2x)\cos(y)\\
                 &+ 0.20 \,\sin(x)\cos(2y)
                 + 0.15 \,\sin(2x)\cos(2y).
\end{split}
\end{equation*}
The corresponding velocity field $u= \frac{\partial \psi}{\partial y}$, $v=-\frac{\partial \psi}{\partial x}$ contains several distinct Fourier modes. When inserted into the quadratic convection term $(u \nabla)u$, these modes interact to generate additional Fourier components. Such nonlinear interactions cannot be absorbed into the pressure gradient. 
This guarantees that the convective parameter $\lambda_1$ influences the evolution and can be identified during training. 

For generating the training data, we fixed the convective parameter at $\lambda_1 = 1.0$ and the viscous parameter at $\lambda_2 = 10^{-1}$. We generated velocity fields for 40 time steps with step size $h_t = 10^{-1}$, which yielded stable and convergent iterations for the chosen spatial discretization.
The resulting dataset consists of complete snapshots of the velocity field, denoted by $(u_{\text{obs}},v_{\text{obs}})$, taken at all 
time intervals, with values stored at all finite difference grid points. 
These grid values are later reused in the FD-PINN loss, where the same finite difference stencils are applied to evaluate spatial and temporal derivatives.

\paragraph{FD-PINN formulation}
We employ a physics-informed neural network that outputs a latent streamfunction
$\psi$ and the pressure $p$. 
The velocity is derived from $\psi$,
\begin{equation*}
    u =  \frac{\partial \psi}{\partial y}, \qquad v = -\frac{\partial \psi}{\partial x},
\end{equation*}
so that incompressibility $\frac{\partial u}{\partial x}+\frac{\partial v}{\partial y}=0$ holds identically. The PDE coefficients
$\lambda_1$ (convection) and $\lambda_2$ (viscosity) are treated as trainable scalars
and optimized jointly with the neural network parameters.

We enforce the momentum residuals written in PDE form as
\begin{align*}
	f &= \frac{\partial u}{\partial t} + \lambda_1 \left(u \frac{\partial u}{\partial x} + v \frac{\partial u}{\partial y}\right) + \frac{\partial p}{\partial x} - \lambda_2 \left(\frac{\partial^2 u}{\partial x^2} + \frac{\partial^2 u}{\partial y^2}\right),\\
    g &=  \frac{\partial v}{\partial t} + \lambda_1 \left(u \frac{\partial v}{\partial x} + v \frac{\partial v}{\partial y}\right) + \frac{\partial p}{\partial y} - \lambda_2 \left(\frac{\partial^2 v}{\partial x^2} + \frac{\partial^2 v}{\partial y^2}\right).
\end{align*}
In the implementation, all derivatives in $f, g$ are evaluated on the periodic training grid by finite differences using the same stencils as in data generation: 
second-order central differences for space and a forward difference in time between consecutive snapshots. 
Since the neural network predicts a streamfunction, obtaining the velocity already requires one spatial derivative, which reduces the available domain by one cell at each boundary. 
The momentum residuals then involve first and second derivatives of $u$ and $v$, which introduce another layer of boundary loss. 
Consequently, residuals can only be enforced on the interior (``core'') grid, excluding two cells at each boundary in both $x$ and $y$. 
The forward difference in time similarly excludes the final snapshot. 

The training objective combines data fidelity and PDE consistency.
Let $\Omega^{h}_1 \subset \Omega^h$ denote the interior grid obtained by removing one cell at each spatial boundary,
and $\Omega^{h}_2\subset \Omega^h$ the interior obtained by removing two cells at each boundary. Let $T^h$ be all
saved time indices and $T^h_{\mathrm{core}} \subset T^h$ those for which a forward difference is
defined (final snapshot excluded). With residuals $f$ and $g$ for the $u$- and $v$-momentum equations,
the training loss is
\begin{equation}\label{eq:NS:loss}
\begin{split}
\frac{1}{N_{\mathrm{data}}} &\sum_{z \in \Omega^{h}_1 \times T^h}
  \left( \left(u(z) - u_{\text{obs}}(z)\right)^2
+ \left(v(z) - v_{\text{obs}}(z)\right)^2 \right)\\
&+ \frac{1}{N_{\mathrm{pde}}}\sum_{z \in \Omega^{h}_2 \times T^h_{\text{core}}}
  \left( f(z)^2 + g(z)^2 \right)
+ \frac{w_{\mathrm{div}}}{N_{\mathrm{div}}}
  \sum_{z \in \Omega^{h}_2 \times T^h}
  \left(\frac{\partial u}{\partial x}(z)+ \frac{\partial v}{\partial y}(z)\right)^2,
  \end{split}
\end{equation}
where $N_{\mathrm{data}} = |\Omega^{h}_1 \times T^h|$, $N_{\mathrm{pde}} = |\Omega^{h}_2 \times T^h_{\text{core}}|$, $N_{\mathrm{div}}= |\Omega^{h}_2 \times T^h|$ are the respective numbers of summands, $w_{\mathrm{div}} = 10^{-3}$, 
and derivatives are evaluated with the same finite-difference stencils as in data generation.
The last term in \eqref{eq:NS:loss}, i.e., the divergence term, is theoretically redundant, as incompressibility holds by construction, but is included for numerical stability. Further, note that pressure is only determined up to an additive constant. Nevertheless, our implementation does not impose any gauge constraint during training. Since only the pressure gradients $\frac{\partial p}{\partial x}$ and $\frac{\partial p}{\partial y}$ appear in the residuals, the recovered pressure is defined only up to a constant shift at each time level.

The neural network consisted of 9 hidden layers with 100 neurons per layer and ReLU activation functions. 
Both the neural network parameters and the PDE coefficients $\lambda_1,\lambda_2$ were treated as trainable variables. 
The FD-PINN was trained for a total of $500\,000$ iterations. 
All experiments were performed on the dataset described above, and the loss function was evaluated on the stencil-compatible interior grid at each iteration.

Although no pressure values were included in the training data, the pressure field is qualitatively well reconstructed; see \cref{Fig:NS:FDPINN} for a visual comparison between the exact and predicted pressure at an intermediate time snapshot. As discussed above, the pressure is determined only up to an additive constant at each time level. In addition, the physical parameters are identified with high accuracy: the recovered values are $\lambda_1 = 0.9546$ and $\lambda_2 = 0.0959$, corresponding to relative errors of approximately $4.54\%$ and $4.10\%$, respectively. 

\begin{figure}[h]
\centering
\newcommand{\mysubfigwidth}{0.3\textwidth}

\subfloat[True pressure]{
    \includegraphics[width=\mysubfigwidth]{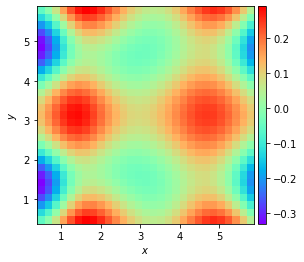}
}\hfill
\subfloat[Solution of FD-PINN]{
    \includegraphics[width=\mysubfigwidth]{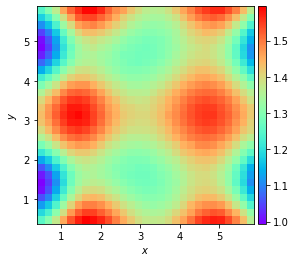}
}\hfill
\subfloat[Solution of FD-PINN on a 5 times finer grid]{
    \includegraphics[width=\mysubfigwidth]{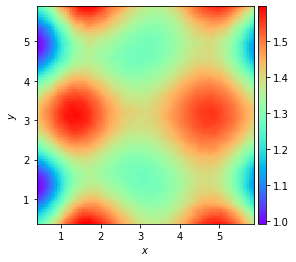}
}

\caption{Solution of the Navier-Stokes equation at an intermediate time snapshot.}
\label{Fig:NS:FDPINN}
\end{figure}

To highlight that the recovery remains robust under noise we corrupt the training data $(u_{\mathrm{obs}}, v_{\mathrm{obs}})$ by adding independent, zero-mean Gaussian noise to each component separately, with the noise standard deviation set to $1\%$ of that component’s own global standard deviation (computed on the original arrays, before adding noise). Specifically,
\[
u_{\mathrm{obs}}^{\mathrm{noisy}} \leftarrow u_{\mathrm{obs}} + \varepsilon_u,
\qquad
v_{\mathrm{obs}}^{\mathrm{noisy}} \leftarrow v_{\mathrm{obs}} + \varepsilon_v,
\]
where
\[
\varepsilon_u \stackrel{\text{i.i.d.}}{\sim} \mathcal{N}\!\big(0,\,(0.01\,\sigma_u)^2\big),
\qquad
\varepsilon_v \stackrel{\text{i.i.d.}}{\sim} \mathcal{N}\!\big(0,\,(0.01\,\sigma_v)^2\big),
\]
and $\sigma_u$ and $\sigma_v$ denote the standard deviation of $u_{\mathrm{obs}}$ and $v_{\mathrm{obs}}$, respectively. The noises added to $u_{\mathrm{obs}}$ and $v_{\mathrm{obs}}$ are statistically independent.

\begin{figure}[h]
\centering
\newcommand{\mysubfigwidth}{0.3\textwidth}

\subfloat[True pressure]{
    \includegraphics[width=\mysubfigwidth]{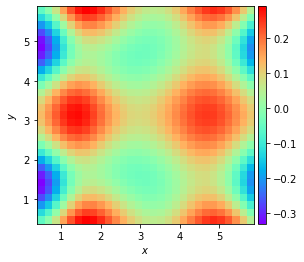}
}\hfill
\subfloat[Solution of FD-PINN]{
    \includegraphics[width=\mysubfigwidth]{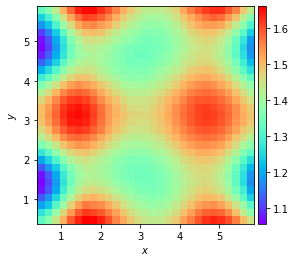}
}\hfill
\subfloat[Solution of FD-PINN on a 5 times finer grid]{
    \includegraphics[width=\mysubfigwidth]{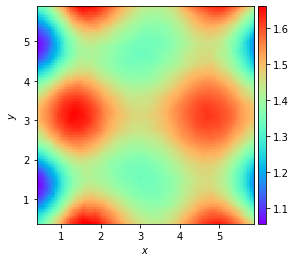}
}

\caption{Solution of the Navier-Stokes equation at an intermediate time snapshot for noisy data.}
\label{Fig:NS:FDPINN:noisy}
\end{figure}

The obtained pressure from the noisy data at the same intermediate time snapshot is depicted in \cref{Fig:NS:FDPINN:noisy} together with the true pressure for comparison reasons. Further the recovered values of the parameters are $\lambda_1 = 0.9522$ and $\lambda_2 = 0.0960$ corresponding to relative errors of approximately $4.75\%$ and $4.00\%$.

\section{Conclusion}\label{Sec:Conclusion}
We analyzed the analytical structure of AD-PINNs and FD-PINNs. 
Under the activation and width assumptions stated in our theory, namely sufficiently regular activation functions and neural networks of sufficient width and depth at least two, we proved that both formulations are ill-posed in the sense of Hadamard: whenever a minimizer exists, there exist in fact infinitely many distinct minimizers; see \cref{thm:nonuniq,thm:nonuniq:FDPINN}. This non-uniqueness persists for any finite-difference stencil and set of collocation points. Thus FD-PINNs do not resolve the ill-posedness of AD-PINNs; they simply exhibit it in a different form.

At the same time, our results (\cref{thm:equivalent:solution,prop:Equivalence:Solution:FD,prop:minimizer:equivalence}) show that whenever the underlying PDE or its finite-difference discretization admits a solution, and provided that $\alpha_\mcD=0$, the corresponding AD-PINN or FD-PINN loss admits a minimizer. 
Further, \cref{prop:Equivalence:Solution:FD,prop:minimizer:equivalence} show that the discrete and neural formulations are tightly coupled. 
In the FD-PINN case, every minimizer $u_\theta \in \mcH$ corresponds to a discrete solution $u_h$ of \eqref{eq:FDM:opt}, defined on the stencil $\Omega^h \cup \Gamma^h$, and the two agree on all stencil points.
If, in addition, the discrete PDE \eqref{PhysicalModel:Discrete} admits a solution and $\alpha_{\mcD}=0$, then any FD-PINN minimizer with zero loss coincides on the stencil with a solution of the discrete PDE. 
In particular, whenever the discrete PDE solution is unique, all zero-loss FD-PINN minimizers induce the same grid values, even though they may differ away from the stencil. 
From this perspective, the ill-posedness of FD-PINNs is confined to their off-stencil behavior: FD-PINNs are non-unique as continuous functions in $\mcH$, but, under uniqueness of the discrete PDE solution, are effectively unique on the grid.

For AD-PINNs, this grid-level uniqueness does not generally hold; see for example \cref{Example:PINN}. 
Instead, \cref{rem:overfitting} shows that the set of AD-PINN minimizers contains an unbounded affine family along which the loss remains minimal while the distance to the true PDE solution can become arbitrarily large in $L^\nu(\Omega)$, for any $\nu\in[1,\infty)$. 
Thus, even exact minimizers of the AD-PINN loss may represent arbitrarily poor approximations of the underlying PDE solution, and identical loss values do not imply comparable prediction quality. 

Taken together, these results reveal a structural contrast: both AD-PINNs and FD-PINNs are ill-posed as function-approximation problems, but FD-PINNs maintain a tight correspondence with the underlying finite-difference scheme. 
In regimes where the discrete PDE admits a unique solution and a zero-loss FD-PINN minimizer exists, all such minimizers agree on the stencil, even though they may differ away from it.
This helps to explain why FD-PINNs often behave more robustly in practice, while also clarifying the limitations of AD-PINNs.

Our numerical experiments confirm these theoretical findings: FD-PINNs can succeed in scenarios where AD-PINNs struggle, such as PDEs with complex boundary geometry. 
The additional stability observed for FD-PINNs is consistent with the fact that, 
whenever the discrete PDE admits a unique solution and the data term vanishes at 
that solution, every zero-loss FD-PINN minimizer must coincide with the discrete 
PDE solution on the finite-difference stencil.  In this regime, the stencil 
values are uniquely determined and mirror those of the classical finite-difference method.

Looking ahead, these findings suggest several directions for improving neural-network-based PDE solvers. While the AD- and FD-PINN formulations do not define a well-posed analytical problem, the FD-PINN shows that introducing additional structure can enforce uniqueness at the numerical level. Developing analogous mechanisms, through regularization, constraints, or modified loss functions that better reflect the stability of the underlying PDE, may help stabilize other PINN approaches as well. Understanding how such design choices shape the optimization landscape represents an important next step toward more reliable neural-network-based methods for PDEs.

\appendix
\section{Auxiliary Results}
\subsection{Closure of Neural Networks Under Linear Combinations}\label{sec:closure}
The set \(\mcH\) is closed under finite linear combinations provided hidden-layer widths may increase: given \(f,g\in\mcH\) and \(c_1,c_2\in\R\), there exists \(\tilde f\in\mcH\) with \(\tilde f=c_1 f+c_2 g\).
Write the parameters of \(f\) and \(g\) as
\[
\begin{aligned}
&f:\quad \varphi^f_{0}(x)=x,\quad \varphi_{i}^f=\sigma\!\big(W_{i}^f \varphi^f_{i-1}+b^f_{i}\big),\quad
f(x)=W^f_{L} \varphi^f_{L-1}+b^f_{L},\\
&g:\quad \varphi^g_{0}(x)=x,\quad \varphi_{i}^g=\sigma\!\big(W_{i}^g \varphi^g_{i-1}+b^g_{i}\big),\quad
g(x)=W^g_{L} \varphi^g_{L-1}+b^g_{L},
\end{aligned}
\]
where \(W_{i}^f\in\R^{d_i^f\times d_{i-1}^f}\), \(b_{i}^f\in\R^{d_i^f}\) and \(W_{i}^g\in\R^{d_i^g\times d_{i-1}^g}\), \(b_{i}^g\in\R^{d_i^g}\) with \(d_i^f, d_i^g\in \N\) for $i=0,\ldots,L\in\N$ and \(d_0^f=d_0^g=d_0\) and \(d_L^f=d_L^g=d_L\).
Construct a depth-\(L\) neural network \(\tilde f\) by running \(f\) and \(g\) in parallel: for \(i=1,\dots,L-1\) set
\[
\widetilde W_i=
\begin{bmatrix}
W_{i}^f & 0\\
0 & W_{i}^g
\end{bmatrix}\in\R^{\tilde d_i\times \tilde d_{i-1}},
\qquad
\widetilde b_i=
\begin{bmatrix}
b^f_{i}\\[2pt] b^g_{i}
\end{bmatrix}\in\R^{\tilde d_i} 
\]
with $\tilde d_i:=d_i^f+d_i^g$, for $i=1,\ldots,L-1$ and $\tilde d_0:=d_0$.
Since \(\sigma\) acts componentwise, \(\widetilde h^{(\ell)}=\begin{bmatrix}h_f^{(\ell)}\\ h_g^{(\ell)}\end{bmatrix}\).
Choose the final (linear) layer as
\[
\widetilde W_{L}=\begin{bmatrix} c_1\,W_{L}^f& c_2\,W_{L}^g\end{bmatrix}\in\R^{d_L\times \tilde d_{L-1}},
\qquad
\widetilde b_{L}=c_1\,b^f_{L}+c_2\,b^g_{L}\in\R^{d_L},
\]
which yields \(\widetilde f(x)=c_1 f(x)+c_2 g(x)\) for all \(x\in\R^{d_0}\).
Consequently, when the output layer is linear, the realizable set forms a vector space (under pointwise operations) up to architectural width, whereas if the output layer is nonlinear (e.g., ReLU, \(\tanh\), softmax) the set is generally not closed under addition.

Analogously one shows under non-width limitation that $\mcH_{\mathrm{reg}}$ is closed under finite linear combinations by noting that if $f,g\in\mcH_{\mathrm{reg}}$ are differentiable at $x\in\R^d$ then also $\tilde{f} = c_1 f + c_2 g$ is differentiable at $x$ and hence $\tilde{f}\in\mcH_{\mathrm{reg}}$.

\subsection{Neural Network Interpolation}
We recall the Hermite interpolation theorem from \cite{LlanasLantaron:17}:
\begin{theorem}[{\cite[Theorem 10]{LlanasLantaron:17}}]\label{thm:ll10}
Let $d\in\N$ and let $\mcA=\{A_1,\dots,A_\ell\}$ be an \emph{admissible} family of $(d-1)$-dimensional affine hyperplanes in $\R^d$, 
that is, every $d$ distinct hyperplanes in $\mcA$ intersect in exactly one point.
Define
\[
A^d := \left\{\, z\in\R^d : z=\bigcap_{A\in \tilde{A}} A \text{ for some } \tilde{A}\subset\mcA,\ |\tilde{A}|=d \,\right\}.
\]
For each $z\in A^d$, define its multiplicity as $m(z):=|\{A\in \mcA : z\in A\}|$ and target data $\zeta_z^\beta\in\R$ for all multiindices $\beta$ with $|\beta|\le m(z)-d$.
Assume $\sigma\in C^{\ell-d}(\R,\R)$ and 
\(
\sigma^{(i)}(0)\neq 0 \quad \text{for all } 0\le i\le \ell-d.
\)
Then there exists a one-hidden-layer neural network $\Phi: \R^d \to \R$ with $\binom{\ell}{d}$ hidden units such that 
\[
D^\beta \Phi(z)=\zeta_z^\beta, \qquad \text{for all } z\in A^d \text{ and } |\beta|\le m(z)-d.
\]
\end{theorem}

To use \cref{thm:ll10} for our purposes we need the following result.
\begin{lemma}[Admissible hyperplanes with prescribed multiplicities]\label{Lem:AdmissiblePlanes}
Let $d\in\N$, $N\ge 1$, and $r\ge 0$. For any set of distinct points
$\{z_i\}_{i=1}^N\subset\R^d$ there exist $\ell=(d+r)N$ affine
$(d-1)$-dimensional hyperplanes $A_1,\dots,A_\ell\subset\R^d$ such that
\begin{enumerate}
\item[(i)]\label{Lem:AdmissiblePlanes:i} $z_i$ lies on exactly $d+r$ of the hyperplanes (hence its multiplicity is $m(z_i)=d+r$) for every $i=1,\dots,N$;
\item[(ii)]\label{Lem:AdmissiblePlanes:ii} the family $\{A_j\}_{j=1}^\ell$ is admissible, i.e., for every $I\subset\{1,\dots,\ell\}$ with $|I|=d$ one has $\bigl|\bigcap_{j\in I} A_j\bigr|=1$.
\end{enumerate}
\end{lemma}

\begin{proof}
For $i=1, \ldots,\ell$ we choose the vectors $v_i = \begin{bmatrix}
1,t_i, t_i^2, \ldots, t_i^{d-1}
\end{bmatrix}^\top$ with distinct $t_i$, i.e., $t_i\not=t_j$ for $i\neq j$. Then any $d$ of these vectors are linearly independent, as they would form a Vandermonde matrix. 
Split the index set $\{1,\dots,\ell\}$ into $N$ disjoint blocks
\[
I_i := \{(i-1)(d+r)+1,\dots,i(d+r)\}, \qquad i=1,\dots,N.
\]
For each $i\in\{1,\ldots,N\}$ and each $j\in I_i$, define
\[
A_j := \{\,\nnew{z}\in\R^d : v_j^\top z = b_j\,\},
\quad\text{where } b_j := v_j^\top z_i.
\]
Then $z_i\in A_j$ for all $j\in I_i$ and property (i) holds by construction.
Now take any subset $I = \{i_1,\ldots,i_d\} \subset \{1,\ldots,\ell\}$. Then the intersection
\(
\bigcap_{j=1}^d A_{i_j}
\)
is the solution of the $d\times d$ linear system $V_I z = b_I$ with $V_I = \begin{bmatrix}
v_{i_1}, \hdots, v_{i_d}
\end{bmatrix}^\top$ and $b_I =\begin{bmatrix}
b_{i_1},\hdots , b_{i_d}
\end{bmatrix}^\top.$ By construction $V_I$ is a Vandermonde matrix and hence invertible, yielding a unique solution $z = \nnew{V}_I^{-1} b_I$. Hence every $d$-tuple of hyperplanes meets in exactly one point, which is precisely the admissibility condition.
\end{proof}

\begin{corollary}[Hermite interpolation with smooth activation]\label{cor:pointwise-bc-vector}
Let $\Omega\subset\R^d$ with boundary $\partial\Omega$ and $N_\mcF,N_\mcB\in\N$. 
Fix finite sets of distinct points $\Omega^h=\{z_\mcF^i\}_{i=1}^{N_\mcF}\subset \Omega$ and $\Gamma^h=\{z_\mcB^j\}_{j=1}^{N_\mcB}\subset\partial\Omega$, and for integers $r_\mcF,r_\mcB\ge 0$ set $\ell:= N_\mcF(d+r_\mcF) + N_\mcB(d+r_\mcB)$.
Let $\sigma\in C^{\ell-d}(\R,\R)$ and assume there exists $a\in\R$ such that $\sigma^{(k)}(a)\neq 0$ for all $0\le k\le \ell-d$.
Then, for any prescribed vectors 
$\mu_{i,\beta}\in\R^c$ and $\eta_{j,\gamma}\in\R^c$, there exists a one-hidden-layer neural network $\Phi: \R^d \to \R^c$ with $c\binom{\ell}{d}$ hidden units such that
\begin{align*}
D^\beta \Phi(z_\mcF^i)&=\mu_{i,\beta}\quad\text{for all }i=1,\dots,N_\mcF,\;|\beta|\le r_\mcF,\\
D^\gamma \Phi(z_\mcB^j)&=\eta_{j,\gamma}\quad\text{for all }j=1,\dots,N_\mcB,\;|\gamma|\le r_\mcB.
\end{align*}

\begin{proof}
Utilizing \cref{Lem:AdmissiblePlanes} we construct an admissible family of affine hyperplanes $\mcA$ in $\R^d$ such that each 
$z_\mcF^i$, $i=1,\ldots,N_\mcF$, lies on exactly $d+r_\mcF$ hyperplanes (so $m(z_\mcF^i)=d+r_\mcF$) and each $z_\mcB^j$, $j=1,\ldots,N_\mcB$, lies on exactly $d+r_\mcB$ hyperplanes (so $m(z_\mcB^j)=d+r_\mcB$).
Since $\mcA$ is an admissible family of affine hyperplanes and the multiplicity for every $z_\mcF^i$ and $z_\mcB^j$ is larger than $d$, we have that $z_\mcF^i,z_\mcB^j\in A^d$ for all $i=1,\ldots,N_\mcF$ and $j=1,\ldots,N_\mcB$, where $A^d$ is defined as in \cref{thm:ll10}.
Apply \cref{thm:ll10} with $\ell=|\mcA|= N_\mcF(d+r_\mcF)+N_\mcB(d+r_\mcB)$, activation $\tilde\sigma(t)=\sigma(t+a) \in C^{\ell-d}(\R,\R)$, and target vectors 
$\zeta_{z_\mcF^i}^\beta=\mu_{i,\beta}$, $\zeta_{z_\mcB^j}^\gamma=\eta_{j,\gamma}$ coordinate-wise\nnew{,} yielding $c$ one-hidden-layer neural networks $\Phi_k : \R^d \to \R$, $k = 1,\ldots, c$ each consisting of \(\binom{\ell}{d}\) hidden units. Stack them to obtain $\Phi:\R^d\to\R^c$ with $c\binom{\ell}{d}$ hidden units. Since $\Phi_k$, $k = 1,\ldots,c$, satisfies the prescribed conditions, the vector-valued map $\Phi=(\Phi_1,\ldots,\Phi_c)$ inherits these properties component-wise, completing the proof.
\end{proof}
\end{corollary}

Here we collect some useful results needed to proof \cref{thm:nonuniq,thm:nonuniq:FDPINN}.
\begin{lemma}[Special interpolation with smooth activation]\label{Lem:Null:Interpolation}
Let $\Omega\subset\R^d$ with boundary $\partial \Omega$ and $L\geq 2$. Fix finite sets $\Omega^h=\{z_\mcF^i\}_{i=1}^{N_\mcF}\subset\Omega$ and 
$\Gamma^h=\{z_\mcB^j\}_{j=1}^{N_\mcB}\subset\partial\Omega$, integers $r_\mcF,r_\mcB\ge 0$, and a target $v\in\R^c$.
Assume the activation $\sigma\in C^{\ell}(\R,\R)$ with $\ell:=N_\mcF (r_\mcF+1)+N_\mcB (r_\mcB+1)$ 
satisfies $\sigma^{(k)}(a)\neq 0$ for $0\le k\le \ell$ and some $a\in\R$. If $L>2$, then we additionally assume that $\sigma$ is strictly monotone. 
Choose $v_*\in\R^d$ such that the projections 
$t_i:=v_*\!\cdot z_\mcF^i\in\R$ and $s_j:=v_*\!\cdot z_\mcB^j\in\R$ are all pairwise distinct for $i=1,\ldots,N_\mcF$ and $j=1,\ldots,N_\mcB$. Then for any $z_0\in{\Omega}\setminus (\Omega^h\cup \Gamma^h)$ such that $v_*\!\cdot z_0 \notin \{t_1,\dots,t_n,s_1,\dots,s_m\}$, there exists a depth-$L$ neural network 
$\Phi:\R^d\to\R^c$ such that 
\begin{align*}
D^\beta \Phi(z_\mcF^i)&=0\quad\text{for all }i=1,\dots,N_\mcF,\;|\beta|\le r_\mcF,\\ 
D^\gamma \Phi(z_\mcB^j)&=0\quad \text{for all }i=1,\dots,N_\mcB,\;|\beta|\le r_\mcB,\qquad
\text{and }\ \Phi(z_0)=v.
\end{align*}
\end{lemma}

\begin{proof}
Define the 1D point set $\mathcal T:=\{t_i\}_{i=1}^n\cup\{s_j\}_{j=1}^m$. Let $a\in\R$ be such that $\sigma^{(k)}(a)\neq 0$ for all $0\leq k\leq \ell$ and define $\tilde{\sigma}(t)=\sigma(t+a)$. Since $\tilde{\sigma}\in C^{\ell}(\R,\R)$ by \cite[Theorem 6]{LlanasLantaron:17} there exists a one-hidden-layer neural network 
$\psi:\R\to\R$ with $\ell+1$ hidden units such that 
\[
\psi^{(k)}(t_i)=0\ \ (i=1,\ldots,N_\mcF,\  0\le k\le r_\mcF),\ \psi^{(k)}(s_j)=0\ \ (j=1,\ldots,N_\mcB,\ 0\le k\le r_\mcB),
\]
and $\psi(t_0)=1$ at a point $t_0\notin\mathcal T$. 
Setting $\Psi(z):=\psi(v_*\!\cdot z)$ yields a one-hidden-layer neural network $\Psi:\R^d \to \R$ such that 
\begin{equation}\label{Eq:Dif:Properties}
\begin{split}
D^\beta \Psi(z_\mcF^i)&=0\quad \text{for all } i=1,\ldots,N_\mcF,\ |\beta|\le r_\mcF,\\
D^\gamma \Psi(z_\mcB^j)&=0\quad \text{for all } j=1,\ldots,N_\mcB,\ |\gamma|\le r_\mcB,\qquad
\text{and }\ \Psi(z_0)=1,
\end{split}
\end{equation}
because $\psi^{(k)}(t_i)=0$ for $0\le k\le r_\mcF$ and $\psi^{(k)}(s_j)=0$ for $0\le k\le r_\mcB$ for all $i,j$.

To construct a depth-$L$ neural network with the same properties we consider a 1D one-hidden-layer neural network $g:\R \to \R$ with the same activation $\sigma$ that satisfies $g(0)=0$ and $g(1)\not=0$. A concrete choice might be $g(t):=\sigma(a t) - \sigma(0)$ with any $a\neq 0$. Since $\sigma$ is strictly monotone, $t=0$ is the only root of $g$. The depth-$L$ neural network is then a composition of $L-1$ functions, i.e.,
\[
\Psi_L:= g \circ g \circ \cdots \circ g \circ \Psi
\]
where $L-2$ copies of $g$ are used. Note that such $L-1$ one-hidden-layer neural networks can be represented by a $(L-1)$-hidden-layer neural network; cf, \cref{fig:hL}. By the Fa\`a di Bruno formula we have for each $z_\mcF^i$, $z_\mcB^j$, and all $|\beta|\le r_\mcF$ and $|\gamma|\le r_\mcB$,
$D^\beta \Psi_L(z_\mcF^i)=0$ and  $D^\gamma \Psi_L(z_\mcB^j)=0$, because \eqref{Eq:Dif:Properties} holds.

\begin{figure}[htb]
\begin{center}

\begin{tikzpicture}[xscale=0.9, yscale=1.3, align=center]
  \tikzstyle{node}=[circle,draw=black]
  \tikzstyle{edge}=[->]
  
  \node[node] (1-1) at (-3, 1) {};
  \node            at (-3,  0) {$\vdots$}; % vertical dots
  \node[node] (1-2) at (-3,-1) { };
  \draw[->, thick, shorten >=2pt] (-4,1) -- (1-1);
  \node[anchor=east] at (-4, 1) {Input $z_1$};
  \draw[->, thick, shorten >=2pt] (-4,-1) -- (1-2);
  \node[anchor=east] at (-4, -1) {Input $z_d$};
  
  \node[node] (2-1) at (-1, 2) {};
  \node[node] (2-2) at (-1, 1) {};
  \node            at (-1,  0) {$\vdots$}; % vertical dots
  \node[node] (2-3) at (-1,-1) {};
  \node[node] (2-4) at (-1,-2) {};
  
  \node[node] (3-1) at ( 1, 1) {};
  \node[node] (3-2) at ( 1,-1) {};
  
  \node[node] (4-1) at ( 3, 1) {};
  \node[node] (4-2) at ( 3,-1) {};
    
  \node            at (4,  1) {$\hdots$}; % horizontal dots
  \node            at (4,  -1) {$\hdots$}; % horizontal dots
  
  \node[node] (5-1) at ( 5, 1) {};
  \node[node] (5-2) at ( 5,-1) {};
  
  \node[node] (6-1) at ( 6,0) {};
  \draw[->, thick, shorten <=2pt] (6-1) -- (7,0);
  \node[anchor=west] at (7, 0) {$\Psi_L(z)$};
  
  \foreach \x in {1,...,1}{
    \foreach \y in {1,...,2}{
    \pgfmathtruncatemacro{\xy}{(\x-1) * 4 + \y}
    {\draw[edge] (1-\x) -- node[ pos = 0.5]{\footnotesize $w_{\xy}$} (2-\y);}
	}
	}
  \foreach \x in {1,...,1}
    \foreach \y in {3,...,3}
    {\draw[edge] (1-\x) -- node[ pos = 0.7]{\footnotesize $w_{N-1}$} (2-\y);}
  \foreach \x in {1,...,1}
    \foreach \y in {4,...,4}
    {\draw[edge] (1-\x) -- node[ pos = 0.6]{\footnotesize $w_{N}$} (2-\y);}

  \foreach \x in {2,...,2}{
    \foreach \y in {1,...,4}{
    {\draw[edge] (1-\x) --  (2-\y);}
	}
	}
	
  \foreach \x in {1,...,2}
    \foreach \y in {1,...,1}
      {\draw[edge] (2-\x) to node[ pos = 0.4]{\footnotesize $a c_\x$} (3-\y);}
  \foreach \x in {3,...,3}
    \foreach \y in {1,...,1}
      {\draw[edge] (2-\x) to node[ pos = 0.4]{\footnotesize $a c_{N-1}$} (3-\y);}
  \foreach \x in {4,...,4}
    \foreach \y in {1,...,1}
      {\draw[edge] (2-\x) to node[ pos = 0.4]{\footnotesize $a c_{N}$} (3-\y);}
       
  \foreach \x in {1,...,4}
    \foreach \y in {2}
      {\draw[edge] (2-\x) to node[ pos = 0.4]{\footnotesize $0$} (3-\y);}

  \foreach \x in {1}
    \foreach \y in {1}
      {\draw[edge] (3-\x) to node[ pos = 0.4]{\footnotesize $a$}(4-\y);} 
  \foreach \x in {2}
    \foreach \y in {1}
      {\draw[edge] (3-\x) to node[pos = 0.4]{\footnotesize $-a$}(4-\y);}
  \foreach \x in {1,...,2}
    \foreach \y in {2}
      {\draw[edge] (3-\x) to node[ pos = 0.4]{\footnotesize $0$}(4-\y);} 
      
  \foreach \x in {1,...,1}
    \foreach \y in {1,...,1}
      {\draw[edge] (5-\x) to node[ pos = 0.6]{\footnotesize $1$} (6-\y);} 
  \foreach \x in {2,...,2}
    \foreach \y in {1,...,1}
      {\draw[edge] (5-\x) to node[ pos = 0.6]{\footnotesize $-1$} (6-\y);}		
		
 \end{tikzpicture}
 
\end{center}
\caption{Illustration of the neural network $\Psi_L:\R^d \to \R$ with input $z=(z_1,\ldots,z_d)^\top\in\R^d$}\label{fig:hL}
\end{figure}
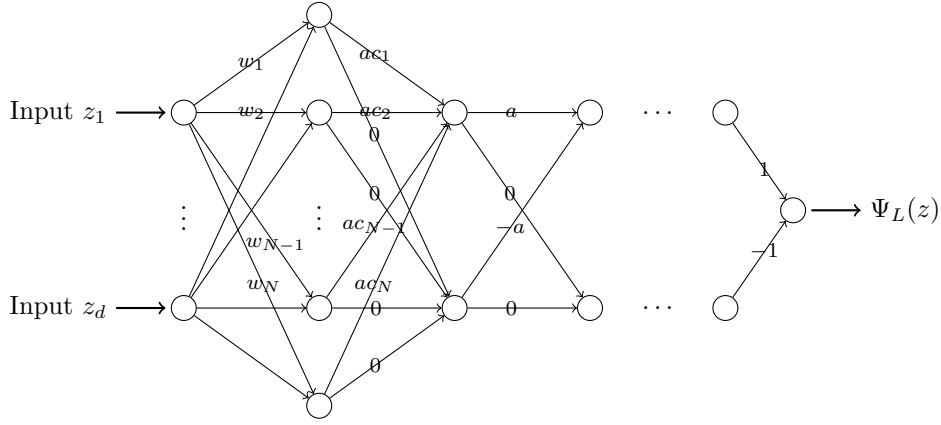
Finally we set the vector-valued neural network as
\[
\Phi(z):= \lambda \Psi_L(z)v\in\R^c
\]
where $\lambda = (g \circ g \circ \cdots \circ g(1))^{-1}$ is well defined, since 0 is the unique root of $g$, so that $\Phi(z_0)=v$.
Thus $\Phi$ is a depth-$L$ neural network and has the required properties.
\end{proof}

\begin{remark}[Choice of projection direction]\label{rem:projection-direction}
In \cref{Lem:Null:Interpolation} we used a direction $v_*\in\R^d$ such that the projections 
$v_*\cdot z_\mcF^i$ and $v_*\cdot z_\mcB^j$ are pairwise distinct. 
To see that such a choice is always possible, fix two distinct points 
$p_i,p_j\in\{z_\mcF^1,\dots,z_\mcF^{N_\mcF},z_\mcB^1,\dots,z_\mcB^{N_\mcB}\}$. 
They collide under projection precisely when $v\cdot(p_i-p_j)=0$, i.e., when 
$v$ lies in the hyperplane
\[
A_{ij} := \{ v\in\R^d : v\cdot(p_i-p_j)=0 \}, \qquad i < j = 2,\ldots, N_\mcF+N_\mcB.
\]
Thus the set of ``bad directions'' is the finite union 
$\mcA = \bigcup_{i<j} A_{ij}$. 
Each $A_{ij}$ is a codimension-one hyperplane through the origin, and therefore 
$\mcA$ has Lebesgue measure zero in $\R^d$. 
Consequently, admissible directions form a dense, full-measure subset of $\R^d$, 
and almost every $v_*$ ensures that all projections are distinct.
\end{remark}

\begin{remark}\label{Rem:widthInterpolationNN}
The depth-$L$ neural network $\Phi$ of \cref{Lem:Null:Interpolation} has the following number of neurons: for $L=2$ we have $d_0=d$, $d_1 = N_\mcF (r_\mcF+1)+ N_\mcB (r_\mcB+1)+1$, $d_2 = c$, while for $L\ge 3$ we get $d_0=d$, $d_1=N_\mcF (r_\mcF+1)+ N_\mcB (r_\mcB+1)+1$, $d_\ell = 2$ for $\ell = 2, \ldots, L-1$, and $d_L=c$ yielding a neural network in $\mcH^M$ with 
\[
M = \sum_{\ell=0}^{L-1} d_{\ell+1}\,(d_\ell + 1) = \begin{cases}
d_1(d+1+c) +c & \text{if $L=2$};\\
d_1(d+3) + 6L+ 3c - 18 & \text{if $L\ge 3$}.
\end{cases}
\]
Note that we count all weights and biases even if they are 0, which is due to the assumption that we consider fully-connected feedforward neural networks.
\end{remark}

\begin{remark}
In the specific setting of \cref{Lem:Null:Interpolation}, the interpolation conditions are structurally simple, which allows us to construct the neural network more directly then in the setting of \cref{cor:pointwise-bc-vector}. Therefore \cref{thm:ll10} is not strictly required, although it provides a convenient general framework that also covers the present case, which would lead to a wider neural network than required for these specialized interpolation conditions.
\end{remark}

\begin{lemma}[ReLU interpolation neural network]\label{Lem:ReLU:Null:interpolation}
Let $\Omega\subset\R^d$ with boundary $\partial \Omega$. Fix finite sets $\Omega^h=\{z_\mcF^i\}_{i=1}^{N_\mcF}\subset\Omega$ and 
$\Gamma^h=\{z_\mcB^j\}_{j=1}^{N_\mcB}\subset\partial\Omega$. Then for any $z_0\in {\Omega}\setminus\left(\Omega^h\cup \Gamma^h \right)$, any target vector $v\in\R^c$ and any $L \ge \lceil \log_2(d+1)  \rceil + 1$ there exists a ReLU neural network $\Phi:\R^d\to\R^c$ of depth $L$ such that $\Phi\equiv 0$ on a neighborhood of each $z_\mcF^i$ in $\Omega$ and on a neighborhood of each $z_\mcB^j$ relative to $\Omega$,
and $\Phi(z_0)=v$. In particular, for every multiindex
$\beta\geq0$ and every order $|\beta|\ge 0$,
\[
D^\beta \Phi(z_\mcF^i)=0\quad i=1,\ldots,N_\mcF, \quad\text{and}\quad D^\beta \Phi(z_\mcB^j)=0 \quad j=1,\ldots,N_\mcB.
\]
\end{lemma}

\begin{proof}
Choose pairwise disjoint open sets $\mcU_i\subset\Omega$ with $z_\mcF^i \in \mcU_i$. 
For each boundary point $z_\mcB^j$, choose an open set $O_j\subset\R^d$ with $z_\mcB^j\in O_j$ and set 
$\mcV_j := O_j\cap\Omega$. Pick 
\[
z_0 \in \Omega \setminus \left(\bigcup_{i=1}^{N_\mcF} \overline{\mcU_i}\ \cup\ \bigcup_{j=1}^{N_\mcB} \overline{\mcV_j}\right).
\]
Construct a continuous and piecewise affine scalar function $\tilde \Psi:\R^d\to\R$ with $\tilde \Psi\equiv 0$ on $\big(\bigcup_{i=1}^{N_\mcF} \mcU_i\big)\cup\big(\bigcup_{j=1}^{N_\mcB} \mcV_j\big)$ and $\tilde \Psi(z_0)=1$ (e.g., a small polyhedral ``tent'' around $z_0$). Define the vector-valued piecewise affine function
$\Psi:\R^d\to\R^c$ by $\Psi(z)=\tilde \Psi(z)\,v$.

Since any piecewise affine function is representable by a ReLU-NN of depth $\lceil \log_2(d+1)  \rceil + 1$ (coordinatewise) \cite[Theorem 2.1]{ArBaMiMu:16}, $\Psi$ is realized by a ReLU-NN.
Because $\Psi\equiv 0$ on each $\mcU_i$ and $\mcV_j$, all classical derivatives vanish at $z_\mcF^i$, and all boundary traces vanish at $z_\mcB^j$. Also $\Psi(z_0)=v$. 
To obtain a depth-$L$ neural network with $L\geq \lceil \log_2(d+1)\rceil + 1$ we just insert layers that implement the identity mapping. This can be realized by $g(t):=\sigma(t)-\sigma(-t) = t$ for all $t\in\R$  and composing the final neural network as $\Phi = g \circ g \circ \cdots \circ g \circ \Psi$ by using $L-\nnew{(\lceil \log_2(d+1)\rceil + 1)}$ $g$'s. This proves the statement. 
\end{proof}

%\section*{Acknowledgments}
%Many thanks to the anonymous reviewers for their valuable suggestions which helped to improve the presentation of the work. 

\bibliographystyle{abbrv}
\bibliography{nmh,PINN_Ref}
\end{document}